\documentclass[11pt,a4paper]{amsart}
\usepackage{graphics}
\usepackage{amscd}
\usepackage{comment}
\usepackage{oldgerm}
\usepackage{amsmath}
\usepackage{amssymb}
\usepackage{amsthm}
\usepackage{color}

\newtheorem{thm}{Theorem}[section]
\newtheorem{prop}[thm]{Proposition}
\newtheorem{lem}[thm]{Lemma}
\newtheorem{cor}[thm]{Corollary}
\newtheorem{dfn}[thm]{Definition}
\newtheorem{fact}[thm]{Fact}
\theoremstyle{remark}
\newtheorem{rem}[thm]{Remark}
\newcommand{\aaa}{\"a}

\newcommand{\Z}{\mathbb{Z}}
\newcommand{\N}{\mathbb{N}}
\newcommand{\R}{\mathbb{R}}
\newcommand{\C}{\mathbb{C}}

\newcommand{\del}{\partial}
\newcommand{\delb}{\overline{\partial}}

\begin{document}

\title[Spectral convergence in geometric quantization]{Spectral convergence in geometric quantization --- the case of non-singular Langrangian fibrations}
\author[K. Hattori]{Kota Hattori}
\author[M. Yamashita]{Mayuko Yamashita}
\address[K. Hattori]{Keio University, 
3-14-1 Hiyoshi, Kohoku, Yokohama 223-8522, Japan}
\email{hattori@math.keio.ac.jp}
\address[M. Yamashita]{Department of Mathematics, Kyoto University, 
606-8502, Kyoto, Japan}
\email{yamashita.mayuko.2n@kyoto-u.ac.jp}
\subjclass[]{}
\maketitle

\begin{abstract}
This paper is a sequel to \cite{hattori2019}. 
We develop a new approach to geometric quantization using the theory of convergence of metric measure spaces. 
Given a family of K\aaa hler polarizations converging to a non-singular real polarization on a prequantized symplectic manifold, 
we show the spectral convergence result of $\delb$-Laplacians, as well as the convergence result of quantum Hilbert spaces. 
We also consider the case of almost K\aaa hler quantization for compatible almost complex structures, and show the analogous convergence results. 
\end{abstract}

\tableofcontents

\section{Introduction}
In this paper, we develop a new approach to geometric quantization using the theory of convergence of metric measure spaces. 
This work is the first step in this project, where we deal with the case of symplectic manifolds admitting non-singular Lagrangian fibrations. 
In the subsequent papers \cite{hattoriyamashitaToric} and \cite{hattori2020}, we deal with more singular settings (toric case and the $K3$-case, respectively). 

On a closed symplectic manifold $(X,\omega)$, 
the prequantum line bundle is a 
triple $(L,\nabla,h)$ of 
a complex line bundle 
$\pi\colon L\to X$ 
equipped with a hermitian 
metric $h$ and a hermitian 
connection $\nabla$ whose curvature form 
$F^\nabla$ is equal to $-\sqrt{-1}\omega$. 
Given a prequantized symplectic manifold $(X, \omega, L, \nabla, h)$, the geometric quantization is a procedure to give a representation of the Poisson algebra consisting of functions on $(X, \omega)$ on a Hilbert space $\mathcal{H}$, called the quantum Hilbert space. 

There are several known ways to construct quantum Hilbert spaces. 
In the approach by Kostant and Souriau, it is given by choosing a polarization on $X$. 
By definition, polarization is an integrable Lagrangian subbundle $\mathcal{P}$ of $TX \otimes \C$, and naively, the quantum Hilbert space $\mathcal{H}$ is thought as the space of sections on $L$ which are covariantly constant along $\mathcal{P}$. 
One fundamental problem in geometric quantization is to find relations among quantizations given by different choices of polarizations. 
In this paper we consider two classes of polarizations, K\aaa hler polarizations and real polarizations, as we now explain.

A K\aaa hler polarization is given by choosing an $\omega$-compatible 
complex structure $J$ on $X = X_J$. 
This gives a polarization $\mathcal{P} = T^{1,0}X_J$. 
In this case $L$ becomes a 
holomorphic line bundle over $X_J$, and
the quantum Hilbert space obtained by this polarization is
$\mathcal{H} = H^0(X_J,L)$, the space of holomorphic sections of $L$. 
On the other hand, a real polarization is given by choosing a Lagrangian fibration $\mu : X^{2n} \to B^n$. 
This gives a polarization $\mathcal{P} = \mathrm{ker}d\mu \otimes \C$. 
Given a Lagrangian fibration, a point $b\in B$ is called a Bohr-Sommerfeld point if the space of parallel sections on $(L, \nabla)|_{\pi^{-1}(b)}$, denoted by $H^0(\pi^{-1}(b); (L, \nabla))$, is nontrivial. 
The set of Bohr-Sommerfeld points, $B_1 \subset B$, is a discrete subset. 
In this case, the quantum Hilbert space is defined by $\mathcal{H}= \oplus_{b \in B_1}H^0(\pi^{-1}(b); (L, \nabla))$. 
More generally we can also use $L^k := L^{\otimes k}$ instead of $L$ in the above, and we get the corresponding quantum Hilbert spaces $\mathcal{H}_k= H^0(X_J; L^k)$ and $\mathcal{H}_k= \oplus_{b \in B_k}H^0(\pi^{-1}(b); (L^k, \nabla))$, 
where $B_k\subset B$ is the set of 
Bohr-Sommerfeld points 
with respect to 
the prequantum bundle $L^k$.  
So the question is to find a relation between these two quantizations, and the problem can be formulated at some different levels. 

The first natural problem is whether the dimensions of $\mathcal{H}$ coincide or not. 
Given a compatible complex structure $J$ and a Lagrangian fibration $\mu$, the equality 
\begin{equation}\label{eq_hol_BS.y}
    \dim H^0(X_J,L^k) = \# B_k
\end{equation}
has been observed in many examples. 
In the case that the Lagrangian fibration is nonsingular, 
the equality (\ref{eq_hol_BS.y}) holds when the Kodaira vanishing holds (see Andersen \cite{andersen1997}, Furuta-Fujita-Yoshida \cite{FFY2010}, and Kubota \cite{Kubota2016}). 
Another example is when $\mu$ is the moment map for a toric symplectic manifold. 
In this case, the base $B$ is a Delzant polytope in $\R^n$, and the set of Bohr-Sommerfeld points is the set of lattice points on the polytope. 
A more nontrivial example includes the case of the moduli space of $SU(2)$-flat connections on a closed surfaces. 
Jeffrey and Weitsman \cite{JW1992} considered real polarizations and a K\aaa hler polarization on this moduli space, and showed that the both sides of the equality (\ref{eq_hol_BS.y}) are given by the same Verlinde formula. 

These interesting phenomena lead us to the next problem: Why they coincide? Can we provide a canonical isomorphism between the quantum Hilbert spaces obtained by two quantizations?
One way to answer this problem is to construct a one-parameter family of $\omega$-compatible complex structures $\{J_s\}_{s >0}$ on $(X, \omega)$ and show that the spaces $H^0(X_{J_s}, L^k)$ converge to the space $\oplus_{b \in B_k}H^0(\pi^{-1}(b); (L^k, \nabla))$ in an appropriate sense. 
This has been worked out in several examples. 
On smooth toric varieties 
with the Lagrangian fibrations given by 
the moment maps, 
Baier, Florentino, Mour\~{a}o and Nunes 
have constructed a one parameter family 
of the pairs of the complex structures 
and the basis of the spaces of 
holomorphic sections of $L$, 
then showed that the holomorphic sections 
converge to the distributional sections 
of $L$ whose support is contained in 
the Bohr-Sommerfeld fibers in \cite{BFMN2011}. 
The similar phenomena were observed in 
the case of the abelian varieties by 
Baier, Mour\~{a}o and Nunes 
in \cite{BMN2010} 
and the flag varieties by Hamilton 
and Konno in \cite{HamiltonKonno2014}.
In these examples, 
the family of complex structures and 
holomorphic sections are described concretely. 
In \cite{yoshida2019adiabatic}, 
Yoshida studied the above phenomena 
in the prequantized symplectic manifolds with the nonsingular 
Lagrangian fibrations by only 
using the local description of 
the almost complex structures.
From the viewpoint of polarizations, the one-parameter families of 
complex structures given in the above papers 
are taken so that the corresponding families of polarizations 
converge to the 
polarizations corresponding to the Lagrangian fibration. 

The purpose of this paper is to give a new approach to this problem using the theory of convergence of metric measure spaces. 
We investigate the 
behavior of the spectrum of $\bar{\del}$-Laplacians, in particular that of the holomorphic sections, 
from the viewpoint of the spectral  
convergence of the Laplace operators on metric measure spaces. 
Here the appropriate notion of convergence is that of spectral structures introduced by Kuwae and Shioya \cite{KuwaeShioya2003}. 
A spectral structure is given by a pair $(H, A)$ of Hilbert space $H$ and a (possibly unbounded) self-adjoint operator $A$. 
There are several types of convergence for a net $\{(H_\alpha, A_\alpha)\}_\alpha$ of spectral structures, and compact convergence is the strongest one. 
In particular compact spectral convergence implies the convergence of spectral set $\sigma(A_\alpha) \to \sigma(A_\infty)$, as well as convergence of eigenspaces in an appropriate sense. 
The most fundamental example is given by Cheeger and Colding \cite{Cheeger-Colding3}. 
They showed the compact convergence of spectral structures of Laplacians, under a measured Gromov-Hausdorff convergence of Riemannian manifolds with a uniform lower bound for Ricci curvatures 
and a uniform upper bound for diameters. 

We now explain our results. 
Denote by $\Delta_{\delb_J}^k$ the $\bar{\del}$-Laplacian on  $L^k$ with respect to the
holomorphic structure  
induced by $J$ and $\nabla$. 
Since we suppose that 
$X$ is closed, we have
$
H^0(X_J,L^k)
= {\rm Ker}\,\Delta_{\delb_J}^k
$. 
The main result of this paper is 
the compact convergence of the family of spectral structures
$\{ (L^2(X,L^k),\Delta_{\delb_{J_s}}^k)\}_{s>0}$ as $s\to 0$ 
to an explicit spectral structure given by a direct sum of that of the Laplacian on the Gaussian space, 
where $\{ J_s\}_s$ is a one-parameter family 
of $\omega$-compatible complex  structures 
whose corresponding polarization converges to a given real polarization. 
Here we suppose $\{ J_s\}_s$ satisfies asymptotically semiflatness defined in Definition \ref{def_asymp_semiflat.y}. 
Under this assumption, 
the diameters of the fibers $\mu^{-1}(b)$ tend to $0$ 
and the distance between the distinct fibers tends to $\infty$ 
as $s\to 0$ with respect to the K\"ahler metrics 
$g_{J_s}=\omega(\cdot,J_s\cdot)$. 
Assuming the semiflatness condition, by Fact \ref{thm_ricci_bound.y}, the Ricci curvatures of $(X, g_{J_s})$ are bounded from below. 
For instance, if 
$\{ g_{J_s}\}_s$ tends to the 
adiabatic limit with normalized volume 
considered in \cite{yoshida2019adiabatic}, 
then it satisfies asymptotically semiflatness.
Moreover, the neighborhood of the nonsingular fiber of 
the {\it large complex structure limit} 
appearing in \cite{GW2000} and \cite{BFMN2011} 
also satisfy asymptotically semiflatness.

Let $(\R^n,{}^tdy\cdot dy,
e^{-k\| y\|^2}d\mathcal{L}_{\R^n})$ be the 
Gaussian space, 
where $\mathcal{L}_{\R^n}$ is the Lebesgue 
measure on $\R^n$ and denote by 
$\Delta_{\R^n}^k$ 
the Laplacian 
of this metric measure space. 
This operator is explicitly written as
\begin{align}\label{eq_gaussian_lap.y}
\Delta_{\R^n}^k \varphi
= \sum_{i=1}^n
\left( -\frac{\del^2 \varphi}{\del y_i^2} 
+2k y_i\frac{\del \varphi}{\del y_i}
\right).
\end{align}
Put 
\begin{align}\label{eq_gaussian_spec_str.y}
H^k 
&:= L^2\left(\R^n, e^{-k\| y\|^2}d\mathcal{L}_{\R^n}\right)\otimes \C. \\ \notag
\end{align}
The main theorem of this paper is the following. 

\begin{thm}\label{thm_main_intro.y}
Let $(X,\omega)$ be a closed 
symplectic manifold of dimension $2n$, $(L,\nabla,h)$ be a prequantum line bundle and $k \ge 1$ be a positive integer. 
Assume that we are given a non-singular Lagrangian fibration $\mu : X \to B$. 
Consider any asymptotically semiflat family of $\omega$-compatible complex structures $\{J_s\}_{s > 0}$. 
Then we have a compact convergence of spectral structures 
\begin{align*}
 (L^2(X,L^k),\Delta_{\delb_{J_s}}^k) \xrightarrow{s \to 0}\bigoplus_{b\in B_k} \left( H^k,  \frac{1}{2}\Delta_{\R^n}^k\right).
\end{align*}
in the sense of Kuwae-Shioya \cite{KuwaeShioya2003}.
\end{thm}

We have a concrete description of the spectrum of the Laplacian on the Gaussian space. 
Namely, it is easy to see that the operator $\Delta^k_{\R^n}$ acting on $H^k$ has a compact resolvent, the set of eigenvalue is $2k\Z_{\ge 0}$ and 
the eigenvalue 
$2kN$ is of multiplicity
$\frac{(N+n-1)!}{(n-1)!N!}$.
Noting the identity $\sum_{p = 0}^N \frac{(p+n-1)!}{(n-1)!p!} = \frac{(N+n)!}{n!N!}$, 
we have the following.

\begin{cor}\label{cor_spec_limit.y}
Under the assumptions in Theorem \ref{thm_main_intro.y}, 
let $\lambda_s^j$ be the $j$-th eigenvalue ($j \ge 1$) of $\Delta_{\bar{\partial}_{J_s}}^k$ acting on $L^2(X; L^k)$, counted with multiplicity. 
For $j \ge 1$, let $N(j) \in \Z_{\ge 0}$ be such that the following inequality is satisfied. 
\begin{align*}
    \# B_k\cdot \frac{(N(j) - 1 + n)!}{n!(N(j)-1)!}
    < j \le \# B_k \cdot \frac{(N(j)  + n)!}{n!(N(j))!}.
\end{align*}
Then we have
\[
\lim_{s \to 0}\lambda_s^j = k\cdot N(j). 
\]
In particular, the number of eigenvalues converging to $0$ is equal to $\# B_k$. 
\end{cor}

However, the compact spectral convergence in Theorem \ref{thm_main_intro.y} is not sufficient to give the desired convergence of quantum Hilbert spaces, because of the possiblility of the existence of nonzero eigenvalues of $\Delta_{\delb_{J_s}}^k$ converging to zero. 
In our second main result, Theorem \ref{thm_conv_hilb2.h}, we show that we have the desired convergence result of quantum Hilbert spaces if $k$ is large enough. 
For the precise statement, see Theorem \ref{thm_conv_hilb2.h}. 
In particular, this means that, for $k$ large enough and $s > 0$ small enough, we get the equality (\ref{eq_hol_BS.y}). 

Now we explain the strategy for the proof of Theorem \ref{thm_main_intro.y}. 
If we have an almost $\omega$-compatible complex structure $J$, it associates a Riemannian metric on $X$ defined by $g_J := \omega(\cdot, J\cdot)$. 
The metric $g_J$, together with the hermitian connection $\nabla$ on $L$, defines a Riemannian metric $\hat{g}_J$ on the frame bundle $S$ of $L$. 
We have a canonical isomorphism
\begin{align*}
    L^2(X, g_J; L^k)\simeq (L^2(S, \hat{g}_J)\otimes \C)^{\rho_k}, 
\end{align*}
where $\rho_k$ is the $S^1$ action given by principal $S^1$-action on $L^2(S, \hat{g}_J)$ and by the formula $e^{\sqrt{-1}t}\cdot z = e^{k\sqrt{-1}t}z$ on $\C$. 
Now suppose that $J$ is integrable. 
Under this isomorphism, we have an identification of operators, 
\begin{align*}
    2\Delta_{\delb_J}^k = \Delta_{\hat{g}_J}^{\rho_k} - (k^2+nk), 
\end{align*}
where $\Delta_{\hat{g}_J}^{\rho_k}$ denotes the metric Laplacian on $(S, \hat{g}_J)$ restricted to the space $(L^2(S, \hat{g}_J)\otimes \C)^{\rho_k}$. 
In this way, we reduce the problem to the analysis of the spectral structure given by $((L^2(S, \hat{g}_J)\otimes \C)^{\rho_k}, \Delta_{\hat{g}_J}^{\rho_k})$. 
So the basic strategy is to consider the family $\{(S, \hat{g}_{J_s})\}_{s > 0}$ of Riemannian manifolds with isometric $S^1$-actions, analyze its Gromov-Hausdorff limit space and guarantee the spectral convergence to the operator on the limit space. 
However, we have $\mathrm{diam}(S, \hat{g}_{J_s}) \to \infty$ in our situation, and this is why we cannot apply the known criteria for spectral convergence directly. 

As for the limit space, we already have the convergence result in \cite{hattori2019}. 
Since the diameter is unbounded, we have to consider the convergence as pointed metric measure spaces. 
For a point $b \in B$, take any lift $u_b \in S$. 
By \cite[Theorem 7.16 and Theorem 1.2]{hattori2019}, we have 
\begin{align}\label{eq_lim_intro_1.y}
    \left\{ \left( S, \hat{g}_{J_s}, 
\frac{\nu_{\hat{g}_{J_s}}}{K\sqrt{s}^n},u_b
\right)\right\}_s
 \xrightarrow{S^1\mathchar`-\mathrm{pmGH}} 
 \left( \R^n\times S^1, g_{k,\infty}, dydt,(0,1)\right) 
\end{align}
if $b\in B_k\setminus (\bigcup_{k'=1}^{k-1} B_{k'})$, 
and 
\begin{align}\label{eq_lim_intro_2.y}
    \left\{ \left( S, \hat{g}_{J_s}, 
\frac{\nu_{\hat{g}_{J_s}}}{K\sqrt{s}^n},u_b
\right)\right\}_s
 \xrightarrow{S^1\mathchar`-\mathrm{pmGH}} 
 \left( \R^n, |dy|^2, dy, 0 \right) 
\end{align}
if $b\notin B_k$ for any positive integer $k$.
Here $K>0$ is some normalizing constant (which does not affect on the spectrum of the Laplacians), and we use the coordinate $y \in \R^n$ and $e^{\sqrt{-1}t}\in S^1$. 
The metric $g_{k, \infty}$ is given by the formula
\begin{align*}
g_{k,\infty}&:=\frac{1}{k^2(1+\| y\|^2)}(dt)^2 
+\sum_{i=1}^n (dy_i)^2. 
\end{align*}
On the right hand sides, the $S^1$ acts on $\R^n$ trivialy and on $\R^n \times S^1$ by the formula $(y,e^{\sqrt{-1}t})\cdot e^{\sqrt{-1}\tau}:=(y,e^{\sqrt{-1}(t+k\tau)})$ for $e^{\sqrt{-1}\tau}\in S^1$. 
If we denote the limit space by $(S_\infty^b, g_\infty^b, \nu_\infty^b, p_\infty^b)$, we see that $(L^2(S^b_\infty)\otimes \C)^{\rho_k} = \{0\}$ if $b\notin B_k$, and if $b\in B_k$, the Laplacian restricted to $(L^2(S^b_\infty)\otimes \C)^{\rho_k}$ is equivalent to $\Delta^k_{\R^n}+(k^2+nk)$ (see Subsection \ref{subsec_limit_sp.h} for detailed explanation, and especially see \eqref{eq_gaussian=eqlap.y}). 
So the Theorem \ref{thm_main_intro.y} is shown 
by the following. 
\begin{thm}\label{thm_main.y}
Let $(X,\omega)$ be a closed 
symplectic manifold of dimension $2n$, $(L,\nabla,h)$ be a prequantum line bundle and $k \ge 1$ be a positive integer. 
Assume that we are given a non-singular Lagrangian fibration $\mu : X \to B$. 
Consider any asymptotically semiflat family of $\omega$-compatible almost complex structures $\{J_s\}_{s > 0}$ and assume that there is a constant $\kappa\in\R$ such that 
${\rm Ric}_{g_{J_s}}\ge \kappa g_{J_s}$. 
Let 
\[
(S_\infty^b, g_\infty^b, \nu_\infty^b, p_\infty^b)
\]
be the pointed $S^1$-equivariant measured Gromov-Hausdorff limit space of the frame bundle $\{(S, \hat{g}_{J_s}, u_b)\}_{s > 0}$ as in \eqref{eq_lim_intro_1.y}. 
Put
\begin{align*}
H_s&=\left(
L^2\left(
S,\frac{\nu_{\hat{g}_{s}}}{K\sqrt{s}^n}
\right)\otimes \C
\right)^{\rho_k},\\
H_\infty&=
\bigoplus_{b\in B_k}
\left( 
L^2(S_\infty^b,\nu_\infty)
\otimes \C
\right)^{\rho_k}, 
\end{align*}
and consider the spectral structures $\Sigma_s$ and $\Sigma_\infty$ associated to the Laplacians restricted on $H_s$ and $H_\infty$, respectively. 
Then we have $\Sigma_s \to \Sigma_\infty$ compactly as $s \to 0$ 
in the sense of Kuwae-Shioya.
\end{thm}

Now we explain how to prove the desired spectral convergence. 
The strong convergence of the spectral structures, which is weaker than the compact convergence, follows easily (Proposition \ref{prop_strong_conv.y}). 
This is a general feature for pointed measured Gromov-Hausdorff convergences with lower bound for Ricci curvatures, and does not require an upper bound for diameters. 
However it is not enough for our purposes; for example a family $\{f_s\}_s$ of normalized eigenfunctions with converging eigenvalues $\{\lambda_s\}_s$, $\lambda_s \to \lambda_\infty$, may not have a convergent subsequene, because the eigenfunctions go away from the basepoint as $s \to 0$. 

In order to show the compact spectral convergence, what we need to show is, roughly speaking, that any family of functions which are $H^{1,2}$-bounded stays close to $\mu^{-1}(B_k)$ as $s \to 0$. 
This is our localization result, Proposition \ref{thm_loc_bs. y}. 
The idea of the localization argument in Section \ref{sec localization.h} comes from the localization argument by Furuta, Fujita and Yoshida \cite{FFY2010}. 
There, they showed a localization result for indices of Dirac-type operator using an ``infinite-dimensional analogue" of Witten deformation, the argument originating from Witten's proof of Morse inequality \cite{witten1982}. 
In our situations, the fiberwise Laplacian of the Lagrangian fibration plays a role of the differential of a Morse function. 

We have so far concentrated on the case where $(X, \omega)$ admits a K\aaa hler structure. 
However, it is not necessarily true that a symplectic manifold admits a K\aaa hler structure. 
Throughout this paper, we consider geometric quantizations on symplectic manifolds with $\omega$-compatible almost complex structures. 
There have been several ways to generalize K\aaa hler quantization to the case where $J$ is not integrable. 
In Section \ref{sec_spec_gap.h} we consider {\it almost K\aaa hler quantization} by Borthwick and Uribe \cite{BU1996}. 
In this approach, we use the operator $\Delta_J^{\sharp k}$ as in Definition \ref{def_alm_lap.y}, and the quantum Hilbert space is given by the eigenspaces which stays bounded as $k\to \infty$. 
It turns out our approach applies to this operator exactly in the same way as in K\aaa hler case. 
In fact, we obtain the spectral convergence 
and the convergence of the quantum Hilbert spaces without the integrability in 
Theorems \ref{thm_main_intro2.h} and 
\ref{thm_conv_hilb.y}, which are the generalization of Theorems \ref{thm_main_intro.y} and 
\ref{thm_conv_hilb2.h} respectively.

The paper is organized as follows. 
In Section \ref{sec Setting.h}, we explain our settings for the problem and recall the result of the previous work of one of the authors in \cite{hattori2019}. In Section \ref{sec_spec_conv.y}, we recall the general notion of spectral convergences and equivariant measured Gromov-Hausdorff convergences, and prove the strong convergence of spectral structures in our settings. 
Section \ref{sec localization.h} is the heart of our proof of the main theorem, where we show the compact convergence by localization argument. 
Combined with the spectral gap result in Sections \ref{sec_spec_gap.h}, \ref{sec_alm.y}, 
this gives the desired picture, namely the space of holomorphic sections converges to the space $\oplus_{b \in B_k}\C$. 
In Section \ref{sec_alm.y}, we also show the examples 
of almost complex structures to which we can apply the main results. 

\section{Settings}\label{sec Setting.h}

Let $(X,\omega)$ be a closed 
symplectic manifold of dimension $2n$ 
and $(L,\nabla,h)$ be a prequantum line bundle, 
that is, 
$(\pi\colon L\to X,h)$ is a complex hermitian line bundle 
and $\nabla$ is a connection on 
$L$ preserving $h$ whose curvature form $F^\nabla$ 
is equal to $-\sqrt{-1}\omega$. 
Put 
\begin{align*}
S := S(L,h) := \{ u\in L;\, h(u,u)=1\}
\end{align*}
and denote by 
$\sqrt{-1}A\in\Omega^1(S,\sqrt{-1}\R)$ 
the connection form corresponding to 
$\nabla$. 
Then $A$ induces a horizontal distribution 
$H=\bigcup_{u\in S} H_u$, where 
$H_u:={\rm Ker}(A_u) \subset T_uS$, 
and $d\pi_u|_{H_u}\colon H_u
\to T_{\pi(u)}X$ gives 
a bundle isomorphism 
$d\pi_H\colon H\to \pi^*TX$.

\subsection{Almost complex structures}\label{sec cpx str.h}
An almost complex structure $J$ is 
{\it $\omega$-compatible} if 
\begin{align*}
\omega(J\cdot,J\cdot)=\omega,\quad 
g_J:=\omega(\cdot,J\cdot) > 0.
\end{align*}
We define a Riemannian metric 
$\hat{g}_J$ on $S(L,h)$ by 
\begin{align*}
\hat{g}_J := A\otimes A
+ (d\pi|_H)^* g_J.
\end{align*}
Note that $S$ is a principal $S^1$-bundle over $X$ 
and the $S^1$-action preserves $\hat{g}_J$. 

Denote by $\Gamma(L)$ the $C^\infty$-sections of $L$ 
and denote by $L^k$ the $k$-times tensor product of $L$. 
Then $L^k$ can be reconstructed as the associate 
bundle $L^k=S\times_{\rho_k} \C$, where 
$\rho_k$ is a $1$-dimensional unitary representation of $S^1$ 
defined by $\rho_k(\sigma)=\sigma^k$ for $\sigma\in S^1$. 
There is the natural identification 
\begin{align}
\Gamma(L^k) &\cong (C^\infty (S)\otimes \C )^{\rho_k}\label{associate.h}\\
&= \left\{ f\colon S \stackrel{C^\infty}{\rightarrow} \C ;\, 
\forall u\in S,\, \forall \sigma\in S^1,\,
\sigma^k f(u\sigma) = 
f(u) \right\}.\notag
\end{align}

Now, the laplace operator 
$\Delta_{\hat{g}_J}$ of $\hat{g}_J$ induces the operator
\begin{align*}
\Delta_{\hat{g}_J}^{\rho_k} \colon 
(C^\infty (S)\otimes \C )^{\rho_k}
\to (C^\infty (S)\otimes \C )^{\rho_k}
\end{align*}
since $S^1$ acts on $(S,\hat{g}_J)$ isometrically. 
Then 
we have 
$\nabla_k^*\nabla_k
= \Delta_{\hat{g}_J}^{\rho_k} -k^2$
under the identification \eqref{associate.h} 
by \cite[Section 3]{Kasue2011}, 
where $\nabla_k$ is the connection on $L^k$ 
induced by $\nabla$.

Next we suppose $J$ is integrable. 
Then $\omega$ is automatically a K\"ahler form 
on the complex manifold $X_J:=(X,J)$, 
and $L^k$ becomes a holomorphic line bundle 
since $F^\nabla$ is of type $(1,1)$. 
Put 
\begin{align*}
\Delta_{\delb_J}^k := (\nabla_{\delb_J})^*\nabla_{\delb_J}
\colon \Gamma(L^k) \to \Gamma(L^k),
\end{align*}
where $(\nabla_{\delb_J})^*$ is the formal adjoint of 
$\nabla_{\delb_J}\colon \Gamma(L^k)\to \Omega^{0,1}(L^k)$. 
By the Bochner-Weitzenb\"ock fromula, 
we have 
\begin{align}\label{eq_delbarlap.y}
2\Delta_{\delb_J}^k = \nabla_k^*\nabla_k-nk
\cong \Delta_{\hat{g}_J}^{\rho_k} - (k^2+nk).
\end{align}
In particular, the space of holomorphic sections 
$H^0(X_J,L^k)$ is identified with 
the $(k^2+nk)$-eigenspace of $\Delta_{\hat{g}_J}^{\rho_k}$.

\subsection{Bohr-Sommerfeld fibers}\label{sec bs fiber.h}
A $C^\infty$ map $\mu$ from $X$ to 
a smooth manifold $B$ of dimension $n$ is called a non-singular
Lagrangian fibration if 
$\mu$ is surjective, all the points in $B$ are regular values 
and $\mu^{-1}(b)$ are Lagrangian submanifolds 
for all $b\in B$. 
It is known that if the fiber is connected and compact, 
then it is diffeomorphic to $n$-dimensional torus 
$T^n$. 
Note that the fibers are always compact since 
we assume that $X$ is compact. 
By the definition of the prequantum line bundle, 
the restriction $L^k|_{\mu^{-1}(b)} \to \mu^{-1}(b)$ is 
a flat complex line bundle. 
\begin{dfn}
\normalfont
$(1)$ For a Lagrangian fibration $\mu\colon X\to B$ 
with connected fibers, $\mu^{-1}(b)$ is a 
{\it Bohr-Sommerfeld fiber of level $k$} if 
$L^k|_{\mu^{-1}(b)} \to \mu^{-1}(b)$ has 
a nonzero flat section. 
$(2)$ $b\in B$ is a 
{\it Bohr-Sommerfeld point of level $k$} if $\mu^{-1}(b)$ is 
a Bohr-Sommerfeld fiber of level $k$. 
$(3)$ $b\in B$ is a 
{\it strict Bohr-Sommerfeld point of level $k$} if $b$ is 
a Bohr-Sommerfeld point of level $k$ 
and never be a Bohr-Sommerfeld point of level $k'$ for any $k'<k$. 
\end{dfn}

\subsection{Polarizations}\label{Polarizations.h}
To treat complex structures 
and Lagrangian fibrations 
uniformly, 
we review the notion of polarizations in the 
sense of \cite{Woodhouse1992}. 

Let $V_\R$ be a real vector space 
of dimension $2n$ with 
symplectic form 
$\alpha\in \bigwedge^2 V_\R^*$ and 
put $V=V_\R\otimes \C$. 
Then $\alpha$ extends $\C$-linearly to 
a complex symplectic form on $V$. 
{\it A Lagrangian subspace} $W${\it of} $V$ 
is a complex vector subspace of $V$ 
such that $\dim_\C W=n$ and 
$\alpha(u,v)=0$ for all $u,v\in W$. 
Put 
\begin{align*}
{\rm Lag}(V,\alpha)
:=\left\{ W\subset V;\, 
W\mbox{ is a Lagrangian subspace}
\right\},
\end{align*}
which is a submanifold of 
Grassmannian ${\rm Gr}(n,V)$. 

For a symplectic manifold $(X,\omega)$, 
put 
\begin{align*}
{\rm Lag}_\omega
:=\bigsqcup_{x\in X} {\rm Lag}(T_x X\otimes \C,\omega_x),
\end{align*}
which is a fiber bundle over $X$. 
A section 
$\mathcal{P}$ of ${\rm Lag}_\omega$ 
is called a {\it polarization of} $X$
if 
\begin{align*}
[\Gamma(\mathcal{P}|_U),\Gamma(\mathcal{P}|_U)]
\subset \Gamma(\mathcal{P}|_U)
\end{align*}
holds for any open set $U\subset X$. 

For instance, 
the subbundle 
\begin{align*}
\mathcal{P}_J:=T^{0,1}_JX \subset TX\otimes \C
\end{align*}
is called a {\it K\"ahler polarization}, 
where $J$ is an $\omega$-compatible 
integrable complex structure. 
In this paper, we also consider 
$\mathcal{P}_J$ with an almost complex structure $J$. 
In this case, $\mathcal{P}_J$ may be non-integrable.

Another example is given by 
Lagrangian fiber bundles. 
Let $\mu\colon X\to B$ 
be a Lagrangian fiber bundle. 
Then 
\begin{align*}
\mathcal{P}_\mu :={\rm Ker}(d\mu)\otimes \C
\subset TX\otimes \C
\end{align*}
is called a {\it real polarization}.

Define $l\colon {\rm Lag}(V,\alpha)\to \{ 0,1,\ldots,n\}$ 
by $l(W):=\dim_\C(W\cap \overline{W})$. 
Then for any K\"ahler polarization $\mathcal{P}_J$ we have 
$l((\mathcal{P}_J)_x)=0$, and for
any real polarization $\mathcal{P}_\mu$ we have 
$l((\mathcal{P}_\mu)_x)=n$. 

Conversely, for a polarization $\mathcal{P}$ 
such that $l(\mathcal{P}_x)=0$ for all $x\in X$, there is a 
unique complex structure $J$ 
such that $\omega(J\cdot,J\cdot)=\omega$ 
and $\mathcal{P}=T^{1,0}_JX$. 
For a polarization $\mathcal{P}$ 
such that $l(\mathcal{P}_x)=n$ for all $x\in X$, 
we obtain the Lagrangian foliation. 

Next we observe the local structure 
of ${\rm Lag}(V,\alpha)$. 
For $W\in {\rm Lag}(V,\alpha)$, 
we can take a basis $\{ w_1,\ldots,w_n\}\subset W$ 
and vectors $u^1,\ldots,u^n\in V$ 
such that 
$\{ w_1,\ldots,w_n,u^1,\ldots,u^n\}$ is a basis of $V$ and 
\begin{align*}
\alpha(w_i,w_j)=\alpha(u^i,u^j)=0,
\quad \alpha(u^i,w_j)=\delta_j^i
\end{align*}
hold. 
Put $W':={\rm span}_\C\{ u^1,\ldots,u^n\}$
and take $A\in {\rm Hom}(W,W')$. 
Then the subspace 
\begin{align*}
W_A:=\left\{ w+Aw\in V;\, w\in W\right\}
\end{align*}
is Lagrangian iff the matrix $( A_{ij})$ defined by 
$Aw_i=A_{ij}u^j$ is symmetric. 
Consequently, we have the identification 
\begin{align}
T_W {\rm Lag}(V,\alpha)
= \left\{ A\in {\rm Hom}(W,W');\, A_{ij}=A_{ji}\right\}.
\label{tangent.h}
\end{align}
Now, we fix $W$ such that $l(W)=n$. 
Then $w_1,\ldots,w_n,u^1,\ldots,u^n$ 
can be taken to be real vectors, hence 
\begin{align*}
l(W_A) =\dim{\rm Ker}(A-\overline{A})= n-{\rm rank}(A-\overline{A})
\end{align*}
holds. 
Moreover $W_A$ comes from an 
almost complex structure which makes $\alpha$ the positive hermitian 
iff 
${\rm Im}A\in M_n(\R)$ is the positive definite symmetric matrix. 
We define 
\begin{align*}
T_W {\rm Lag}(V,\alpha)_+
:= \left\{ A\in {\rm Hom}(W,W');\, A_{ij}=A_{ji},\, {\rm Im}A>0\right\}
\end{align*}
under the identification \eqref{tangent.h}.
If $W_t$ is a smooth curve in 
${\rm Lag}(V,\alpha)$ such that 
$l(W_0)=n$ and 
$\frac{d}{dt}W_t|_{t=0}\in T_{W_0} {\rm Lag}(V,\alpha)_+$, 
then there is $\delta>0$ such that 
$l(W_t)=0$ and $\alpha(w,\bar{w})>0$ for 
any $w\in W_t\setminus \{ 0\}$ and 
$0<t\le \delta$. 
Conversely, 
even 
if $W_t$ satisfies $l(W_0)=n$ and 
\begin{align*}
l(W_t)=0,\quad 
\alpha(w,\bar{w})>0 \mbox{ for 
any }w\in W_t\setminus \{ 0\}
\end{align*}
for all $t>0$, 
$\frac{d}{dt}W_t|_{t=0}$ is not 
necessary to be in 
$T_{W_0} {\rm Lag}(V,\alpha)_+$ 
since the closure of positive definite 
symmetric matrices contains semi-positive 
definite symmetric matrices.

From now on we consider one parameter 
families of $\omega$-compatible almost complex structures 
$\{ J_s\}_{0<s<\delta}$ on 
$(X,\omega)$. 
We assume the following condition $\spadesuit$ 
for $\{ J_s\}$. 
Let ${\rm pr}\colon X\times [0,\delta)\to X$ 
be the projection and 
${\rm pr}^*{\rm Lag}_\omega$ be the 
pullback bundle. 
\begin{itemize}
\setlength{\parskip}{0cm}
\setlength{\itemsep}{0cm}
 \item[$\spadesuit$] 
There is a smooth section 
$\mathcal{P}$ of ${\rm pr}^*{\rm Lag}_\omega
\to X\times [0,\delta)$ such that 
$\mathcal{P}(\cdot,s)=\mathcal{P}_{J_s}|_U$ for 
$s>0$, $\mathcal{P}(\cdot,0)
=\mathcal{P}_\mu|_U$ and 
\begin{align*}
\frac{d}{ds}\mathcal{P}(x,s)\Big|_{s=0}\in 
T_{\mathcal{P}_\mu(x)}{\rm Lag}(T_xX\otimes \C, 
\omega_x)_+
\end{align*}
for any $x\in X$. 
\end{itemize}

\subsection{Local descriptions}\label{sec local.h}
Here, we describe $\omega$-compatible 
almost complex structure 
$J$, hermitian metric 
$g_J=\omega(\cdot,J\cdot)$ and the 
metric $\hat{g}_J$ locally under the 
action-angle coordinate. 
For any $b\in B$ there is a contractible open neighborhood 
$U\subset B$ of $b$ and 
action-angle coordinate 
\begin{align*}
(x,\theta)=(x_1,\ldots,x_n,\theta^1,\ldots,\theta^n)
\end{align*}
on $X|_U:=\mu^{-1}(U)\cong U\times T^n$. 
Then we may write 
\begin{align*}
\omega|_{X|_U} &= dx_i\wedge d\theta^i,\\
\mu &= x.
\end{align*}

If $\mathcal{P}_J$ is close to 
$\mathcal{P}_\mu$ as polarizations,
then the frame of $T^{0,1}_JX$ on 
$X|_U$ is given by 
\begin{align}\label{eq_local_hol_frame.y}
\frac{\partial}{\partial\theta^i}
+\bar{A}_{ij}(x,\theta) \frac{\partial}{\partial x_j},\quad 
i=1,\ldots,n
\end{align}
for some 
\begin{align*}
A(x,\theta)=\left( A_{ij}(x,\theta)\right)_{i,j}\in 
C^\infty(X|_U)\otimes M_n(\C). 
\end{align*}
Then the $\omega$-compatibility 
of $J_s$ is equivalent to
\begin{align*}
A_{ij}=A_{ji},\quad 
{\rm Im}A>0.
\end{align*}
Conversely, if a complex matrix valued function 
$A$ 
satisfies above properties, 
then we can recover 
the $\omega$-compatible alomost 
complex structure $J$ on $X|_U$. 
The integrability of $J$ 
is equivalent to 
\begin{align}
\frac{\partial A_{jk}}{\partial\theta^i} 
- \frac{\partial A_{ik}}{\partial\theta^j} 
+A_{il}\frac{\partial A_{jk}}{\partial x_l} 
-A_{jl}\frac{\partial A_{ik}}{\partial x_l} = 0.
\label{integrable.h}
\end{align}

If we put $A_{ij}=P_{ij}+\sqrt{-1}Q_{ij}$, where $P_{ij},Q_{ij}\in C^\infty(X|_U; \R)$, 
and denote by $(Q^{ij})$ the inverse of $(Q_{ij})$, 
then one can see 
\begin{align}\label{eq_met.y}
g_J|_U &= 
g_A:= (Q_{ij} + P_{ik}Q^{kl}P_{lj}) d\theta^i d\theta^j
-2 P_{ik}Q^{jk} d\theta^i dx_j
+ Q^{ij} dx_i dx_j.
\end{align}

Next we describe $(L,\nabla,h)$ on $X|_U$. 
Since the first Chern class of 
$L|_{X|_U}$ vanishes, it is trivial as a $C^\infty$-hermitian line bundle. 
The identification 
\begin{align*}
L|_{X|_U} = X|_U\times \C
\end{align*}
satisfying 
$h((p,w),(p,w)) = |w|^2$ for $(p,w)\in X|_U\times \C$ 
is given by a smooth section $E\in \Gamma(L|_{X|_U})$ 
with $h(E,E)\equiv 1$. 
Put 
\begin{align*}
S|_U:=S(L|_{X|_U},h) = X|_U \times S^1. 
\end{align*}
The points in $X|_U \times S^1$ 
are written as 
$(x,\theta,e^{\sqrt{-1}t})$, 
which give the coordinate on $S|_U$. 

Let $\gamma_i\in H_1(\mu^{-1}(b))$ 
be the homology class represented by 
\begin{align*}
\{ (b,0,\ldots,0,\theta^i,0\ldots,0);\, 
0\le \theta^i\le 2\pi\}.
\end{align*}
Denote by $e^{\sqrt{-1}a_i}\in S^1$ the 
element of the holonomy group ${\rm Hol}(L,\nabla)$ 
generated by $\gamma_i$. 
Now, performing a parallel  translation on the base if necessary, we take an action-angle coordinate 
such that $x_i(b)\equiv - \frac{1}{2\pi}a_i \pmod{\Z}$. 
\begin{prop}
Under some local trivializations 
$L|_{X|_U} = X|_U \times \C$, 
$\nabla=d-\sqrt{-1}x_i d\theta^i$ holds, 
where $d$ is the connection such that $p\mapsto (p,0)
\in X|_U \times \C$ is a flat section. 
In particular, 
\begin{align*}
\hat{g}_{J}|_{S|_U}
= (dt-x_i d\theta^i)^2 +g_A
\end{align*}
holds. 
\end{prop}
\begin{proof}
Fix $E\in\Gamma(L|_{X|_U})$ such that $h(E,E)=1$. 
Then we have 
$\nabla E =\sqrt{-1}\alpha\otimes E$ 
for some $\alpha\in\Omega^1(X|_U)$. 
Since 
\begin{align*}
F^\nabla=\sqrt{-1}d\alpha=-\sqrt{-1}\omega,
\end{align*}
then we may write 
\begin{align*}
\alpha =-x_id\theta^i + \alpha'
\end{align*}
for some closed form $\alpha'\in\Omega^1(X|_U)$. 
Since 
\begin{align*}
H^1(X|_U) = {\rm span}\{ 
d\theta^1,\ldots, d\theta^n \},
\end{align*}
hence there 
are constants $\alpha_i\in \R$ and 
a smooth function $f\in C^\infty(X|_U)$ such that 
\begin{align*}
\alpha =-x_id\theta^i + \alpha_id\theta^i +df.
\end{align*}
Since we can see 
$a_i = \int_{\gamma_i}\alpha = 2\pi(-x_i(b)+\alpha_i)$ 
modulo $2\pi\Z$, 
hence $\alpha_i\in\Z$ holds. 
If we put $E' := e^{-\sqrt{-1}(\alpha_i\theta^i + f)}E$, 
then we have 
\begin{align*}
\nabla E' = -\sqrt{-1}x_id\theta^i\otimes E',
\end{align*}
therefore, $\nabla=d-\sqrt{-1}x_id\theta^i$ 
holds by the trivialization 
$L|_{X|_U} = X|_U \times \C$ given by $E'$. 
By the argument in \cite[Section 3]{hattori2019}, 
we have the local description of 
$\hat{g}_J$. 
\end{proof}

Given a family $\{J_s\}_s$ of $\omega$-compatible (almost) complex structures, denote by $A(s,\cdot)$ the local 
description of $J_s|_{X|_U}$. 
For simplicity, we often write 
$A=A(s,\cdot)$ if there is no fear 
of confusion. 
By assuming 
$\spadesuit$, 
there are a constant $K>0$ and 
$A^0\in C^\infty(X|_U)\otimes M_n(\C)$ 
such that 
$\sup_{i,j}\| A_{ij}(s,\cdot)-sA^0_{ij}
\|_{C^2(X|_U)} \le Ks^2$ and 
$\sup_{i,j}\| A^0_{ij}\|_{C^2(X|_U)}<\infty$. 
Moreover, 
we can show that 
\begin{align*}
A^0_{ij}=A^0_{ji},\quad 
{\rm Im}(A^0)>0.
\end{align*}

By putting $A^0=P^0+\sqrt{-1}Q^0$ and 
$\Theta^0 = Q^0+P^0(Q^0)^{-1}P^0$, 
we have 
\begin{align}\label{eq_metric.y}
g_{J_s}|_{X|_U}
&= s\cdot {}^t d\theta 
\left\{ 
\Theta^0
+O(s) \right\}  d\theta
- 2 {}^t d\theta
\left\{ 
P^0(Q^0)^{-1} +O(s)\right\} dx\\
&\quad\quad 
+ \frac{1}{s}\cdot 
{}^t dx \left\{ (Q^0)^{-1}+O(s)\right\} dx.\notag
\end{align}
Consider the following condition for the family $\{J_s\}_{s}$. 
\begin{dfn}\label{def_asymp_semiflat.y}
\normalfont
A family of $\omega$-compatible almost 
complex structures $\{J_s\}_{0<s<\delta}$ satisfying $\spadesuit$ is called an {\it asymptotically semiflat family} if ${\rm Im}(A^0)$ is independent of $\theta$ in the local description \eqref{eq_metric.y}. 
\end{dfn}
This definition does not depend on the choice of action-angle coordinate. 
This condition is equivalent to the lower-boundedness of Ricci curvatures for $\{g_{J_s}\}_{0 < s < \delta}$ 
if we assume that $J$ is integrable. 
Namely, we have the following. 
\begin{fact}[{\cite[Proposition 7.6]{hattori2019}}]\label{thm_ricci_bound.y}
Assume a family of $\omega$-compatible complex structures $\{ J_s\}_{0<s<\delta}$ satisfies $\spadesuit$. 
Then, there is $\kappa\in\R$ such that 
${\rm Ric}_{g_{J_s}} \ge \kappa g_{J_s}$ for any 
$s>0$, if and only if 
$\{J_s\}_{0<s<\delta}$ is an asymptotically semiflat family. 
\label{ric_semiflat.h}
\end{fact}

\begin{rem}
Actually, it is possible to weaken the assumption $\spadesuit$ for the family of $\omega$-compatible complex structures $\{J_s\}_s$. 
All we need for our argument below are the convergence of the frame bundle as in Fact \ref{mGH conv.h.} and the lower-boundedness of Ricci curvatures as in Fact \ref{thm_ricci_bound.y}. 
However, without assuming $\spadesuit$, the condition for lower-boundedness of Ricci curvatures becomes complicated.
To avoid this technical difficulty, in the below we work under the asymptotically semiflatness assumption (in particular the condition $\spadesuit$).  
\end{rem}

\subsection{The structure of the limit spaces}\label{subsec_limit_sp.h}
Now we review the result in \cite{hattori2019}, on the pointed $S^1$-equivariant measured Gromov-Hausdorff limits of $\{(S, \hat{g}_s)\}_{0 < s < \delta}$ under the deformation satisfying $\spadesuit$. 
Let $g_{k,\infty}$ and $\nu_\infty$ 
be a Riemannian metric and a measure 
on $\R^n\times S^1$ defined by 
\begin{align}\label{eq_lim_met.y}
g_{k,\infty}&:=\frac{1}{k^2(1+\| y\|^2)}(dt)^2 
+\sum_{i=1}^n (dy_i)^2,\\ \notag
d\nu_\infty&:=dy_1\cdots dy_ndt,
\end{align}
where $k$ is a positive integer, 
$y=(y_1,\ldots,y_n)\in\R^n$ 
and $e^{\sqrt{-1}t}\in S^1$. 
We define the isometric $S^1$-action 
on 
$(\R^n\times S^1,g_{k,\infty},\nu_\infty)$ 
by $(y,e^{\sqrt{-1}t})\cdot e^{\sqrt{-1}\tau}:=(y,e^{\sqrt{-1}(t+k\tau)})$ for $e^{\sqrt{-1}\tau}\in S^1$. 
The followings are the main results of \cite{hattori2019}. 
\begin{fact}[{\cite[Theorem 7.16]{hattori2019}}]
In the above situations, assume we are given an 
asymptotically semiflat family of almost complex structures $\{J_s\}_{0<s<\delta }$ satisfying the condition $\spadesuit$. 
Let $b\in B$, $k$ be a positive integer 
and fix 
$u_b\in (\pi\circ\mu)^{-1}(b)$. 
Assume that $\mu^{-1}(b)$ is a Bohr-Sommerfeld 
fiber of level $k$ and 
not a Bohr-Sommerfeld 
fiber of level $k'$ for any $0<k'<k$. 
Let $\nu_{\hat{g}_{J_s}}$ be 
the Riemannian measure 
of $\hat{g}_{J_s}$. 
We assume that there is 
$\kappa\in \R$ such that 
${\rm Ric}_{g_{J_s}}
\ge \kappa g_{J_s}$ holds for all $0<s<\delta$. 
Then for some positive constant 
$K>0$, 
the family of pointed metric measure spaces with 
the isometric $S^1$-action 
\begin{align*}
\left\{ \left( S, \hat{g}_{J_s}, 
\frac{\nu_{\hat{g}_{J_s}}}{K\sqrt{s}^n},u_b
\right)\right\}_s
\end{align*}
converges to $\left( \R^n\times S^1, g_{k,\infty}, \nu_{\infty},(0,1)\right)$ as $s\to 0$ 
in the sense of the pointed $S^1$-equivariant 
measured Gromov-Hausdorff topology. 
\label{mGH conv.h.}
\end{fact}

\begin{rem}
\normalfont
We should remark that 
\cite[Theorem 7.16]{hattori2019} 
was proven under the assumption that 
$J_s$ are 
integrable. 
However, we can replace the integrability of $J_s$ 
by the asymptotic semiflatness of $\{ J_s\}$ 
without any change of the proof. 
\end{rem}

The Laplacian $\Delta_{k,\infty}$ 
on the metric measure space $(\R^n\times S^1,g_{k,\infty},\mu_\infty)$
is defined so that 
\begin{align*}
\int_{\R^n\times S^1}
(\Delta_{k,\infty} f_1) f_2 d\mu_\infty
= \int_{\R^n\times S^1}\langle 
df_1,df_2\rangle_{g_{k,\infty}} d\mu_\infty
\end{align*}
holds for any $f_1,f_2\in C_c^\infty(\R^n\times S^1)$ (see \cite[p.3]{Honda2018}). 
We have 
\begin{align*}
\Delta_{k,\infty} f
= \Delta_{\R^n} f -k^2(1+ \| y\|^2) \frac{\del^2 f}{\del t^2},
\end{align*}
where $\Delta_{\R^n}=-\sum_{i=1}^n\frac{\del^2}{\del y_i^2}$. 

The relation between the above operator and the Laplacian on the Gaussian space, $(H^j, \Delta_{\R^n}^j)$ in \eqref{eq_gaussian_lap.y} and \eqref{eq_gaussian_spec_str.y}, is explained as follows. 
Let us fix $k$, and in the rest of this subsection, $\R^n \times S^1$ denotes the limit space at strict Bohr-Sommerfeld point of level $k$. 
For a positive integer $j\in k\Z$, if we write $j = kl$ we have
\begin{align*}
\left( L^2 (\R^n\times S^1)
\otimes \C\right)^{\rho_{j}}
= \left\{ \varphi(y)e^{-\sqrt{-1}l t};\, \varphi\in L^2 (\R^n)\right\}
\end{align*}
This induces the isomorphism 
\begin{align*}
L^2(\R^n,e^{-j\| y\|^2}d\mathcal{L}_{\R^n})
\otimes\C
&\cong 
\left( L^2 (\R^n\times S^1,d\mu_\infty)
\otimes \C\right)^{\rho_{j}} \\
\varphi &\mapsto \varphi \cdot e^{-\frac{j\|y\|^2}{2}  -\sqrt{-1}lt}
\end{align*}
and the identification of the operators 
\begin{align*}
\Delta_{\R^n}^j
\cong 
\Delta_{k,\infty}^{\rho_j} - (j^2+ jn).
\end{align*}
In this way, we identify the spectral structures, 
\begin{align}\label{eq_gaussian=eqlap.y}
    (H^j, \Delta_{\R^n}^j) \cong \left( \left(L^2(\R^n \times S^1),d\mu_\infty)
\otimes \C\right)^{\rho_{j}}, \Delta_{k,\infty}^{\rho_j} - (j^2+ jn)\right)
\end{align}

On the other hand, $\left( L^2 (\R^n\times S^1)
\otimes \C\right)^{\rho_{j}}=\{ 0\}$ if $j\notin k\Z$. 

Since the spectrum and the eigenspaces of the operator $\Delta_{\R^n}^j$ on $H^j$ is well-known, by \eqref{eq_gaussian=eqlap.y} we have the following eigenspace decompositions for these spaces. 

\begin{fact}[{\cite[Theorem 8.1]{hattori2019}}]\label{fact_eig.y}
Let $l\in\Z_{>0}$, $j=kl$ and 
\begin{align*}
W(j,\lambda) := \left\{ 
f\in \left( C^\infty (\R^n\times S^1)
\otimes \C\right)^{\rho_j};\,
\left( \Delta_{k,\infty} - j^2 - jn\right)f = 2\lambda f \right\}.
\end{align*}
Then there is an orthogonal decomposition 
\begin{align*}
(L^2(\R^n\times S^1)\otimes \C)^{\rho_j}
&=\overline{ \bigoplus_{d\in\Z_{\ge 0}}W(j,jd)},
\end{align*}
where 
\begin{align*}
W(j,jd)
&= {\rm span}_\C\left\{ e^{\frac{j\| y\|^2}{2}-\sqrt{-1}lt}
\left(\frac{\del}{\del y}\right)^N
(e^{-j\| y\|^2});\, 
N\in(\Z_{\ge 0})^n,\, 
|N|=d\right\}.
\end{align*}
\end{fact}

\section{Spectral convergence}\label{sec_spec_conv.y}

\subsection{Convergence of spectral structures}\label{subsec_spec_conv_general.y}
In \cite{KuwaeShioya2003}, 
Kuwae and Shioya introduced the notion of 
spectral structures for the Laplacian which enabled us 
to treat the convergence of eigenvalues in the systematic way. 
In this subsection we review the framework developed 
in \cite{KuwaeShioya2003}. 
In this paper, Hilbert spaces are always assumed to be separable, and to be over $\mathbb{K}=\R$ 
or $\C$. 

Let $\mathcal{A}$ be a directed set, and let us fix an element $\infty \in \mathcal{A}$. 
The typical examples used in this paper are $\mathcal{A} = \Z_{>0} \sqcup \{\infty\}$ and $\mathcal{A} = \R_{\ge 0}$ with $0 \in \R_{\ge 0}$ regarded as the element $\infty \in \mathcal{A}$. 

\begin{dfn}
\normalfont
Let $\{H_\alpha\}_{\alpha \in \mathcal{A}}$ be a net of Hilbert spaces. 
We say the net $\{H_\alpha\}_\alpha$ {\it converges to} $H_\infty$, or $\{H_\alpha\}_\alpha$ is a {\it convergent net of Hilbert spaces} if 
it is equipped with a dense subspace $\mathcal{C}\subset H_\infty$ 
and linear operators
$\Phi_\alpha\colon \mathcal{C}\to H_\alpha$ 
which satisfy 
\begin{align}
\lim_{\alpha\to \infty}\| \Phi_\alpha(u)\|_{H_\alpha}=\| u\|_{H_\infty}
\label{conv str}
\end{align}
for any $u\in\mathcal{C}$. 
\label{def_conv_Hilb.h}
\end{dfn}

\begin{dfn}[{\cite[Definition 2.4 and 2.5]{KuwaeShioya2003}}]
\normalfont
Let $\{H_\alpha\}_{\alpha \in \mathcal{A}}$ be a convergent net of Hilbert spaces and assume we are given $u_\alpha\in H_\alpha$ for $\alpha\in\mathcal{A}$. 
\begin{itemize}
\setlength{\parskip}{0cm}
\setlength{\itemsep}{0cm}
 \item[(1)] A net $\{ u_\alpha\}_\alpha$ {\it converges to $u_\infty$ strongly} as $\alpha\to \infty$ 
if there exists a net $\{ \tilde{u}_\beta\}_{\beta\in \mathcal{B}} \subset H_\infty$ tending to 
$u_\infty$ such that 
\begin{align*}
\lim_{\beta}\limsup_{\alpha\to \infty}\| \Phi_\alpha(\tilde{u}_\beta) - u_\alpha\|_{H_\alpha} = 0.
\end{align*}
 \item[(2)] A net $\{ u_\alpha\}_\alpha$ {\it converges to $u_\infty$ weakly} as $\alpha\to \infty$ 
if \begin{align*}
\lim_{\alpha\to \infty}\langle u_\alpha,v_\alpha\rangle_{H_\alpha} 
= \langle u_\infty,v_\infty\rangle
\end{align*}
holds for any net $\{v_\alpha\}_{\alpha \in \mathcal{A}}$ such that 
$v_\alpha\to v_\infty$ strongly. 
\end{itemize}
\label{def spec}
\end{dfn}

Next we define the notion of convergence of bounded operators. 
Suppose $\{H_\alpha\}_{\alpha \in \mathcal{A}}$ is a convergent net, and we have a net of bounded operators $\{B_\alpha \in L(H_\alpha)\}_{\alpha \in \mathcal{A} }$.
\begin{dfn}[{\cite[Definition 2.6]{KuwaeShioya2003}}]\label{def_conv_op.y}
\normalfont
We say that a net $\{B_\alpha\}_{\alpha\in \mathcal{A}}$ {\it strongly converges to} $B_\infty$ if $B_\alpha u_\alpha \to B_\infty u_\infty$ strongly for any sequence $\{u_\alpha\}_{\alpha \in \mathcal{A}}$ with $u_\alpha \in H_\alpha$ strongly converging to $u_\infty \in H_\infty$.  
We say that $\{B_\alpha\}_{\alpha\in \mathcal{A}}$ {\it compactly converges to} $B_\infty$ if $B_\alpha u_\alpha \to B_\infty u_\infty$ strongly for any sequence $\{u_\alpha\}_{\alpha \in \mathcal{A}}$ with $u_\alpha \in H_\alpha$ weakly converging to $u_\infty \in H_\infty$.  

\end{dfn}
Note that when $B_\alpha \to B_\infty$ compactly, $B_\infty$ is necessarily a compact operator. 

Next, we define the notion of spectral structure, which is crucial in our paper. 
\begin{dfn}\label{def_specstr.y}
\normalfont
A {\it spectral structure} is a pair $(H, A)$, where $H$ is a Hilbert space and $A \colon \mathcal{D}(A) \to H$ is a densely defined self-adjoint linear operator on $H$. 
We say that a spectral structure $(H, A)$ is {\it positive} if $A$ is a nonnegative operator. 
\end{dfn}

\begin{rem}\label{rem_specstr.y}
\normalfont
The readers should note that the notion of spectral structure defined in Definition \ref{def_specstr.y} is more general than that in \cite[Section 2.6]{KuwaeShioya2003}; their definition corresponds to {\it positive} spectral structures in Definition \ref{def_specstr.y}. 
More precisely, for a Hilbert space $H$, they define a spectral structure on $H$ to be a set of data
\begin{align*}
\Sigma := (A,\mathcal{E},E,\{ T_t\}_{t\ge 0}, \{ R_\zeta\}_{\zeta\in\rho(A)}), 
\end{align*}
where $A$ is a densely defined positive selfadjoint operator on $H$ 
which is called the infinitesimal 
generator, 
$\mathcal{E}$ is a quadratic form associated with $A$, $E$ is the spectral measure of $A$, 
$T_t:=e^{-tA}$, $R_\zeta=(\zeta-A)^{-1}$ and 
$\rho(A)$ is the resolvent set of $A$. 
However, the data above is completely determined only by the operator $A$, so their spectral structures are in one to one correspondence with positive spectral structures in our paper. 
Since we need to consider spectral convergence of operators which are not necessarily positive in Section \ref{sec_alm.y}, we generalize the notion as above. 
\end{rem}

If we have a spectral structure $(H_\alpha, A_\alpha)$, for a Borel subset $I \subset \mathbb{R}$, let $E_\alpha(I) \in B(H_\alpha)$ be the corresponding spectral projection of the selfadjoint operator $A_\alpha$ on $H_\alpha$. 
Let us define $n_\alpha(I) := \dim E_\alpha(I)H_\alpha \in \Z_{\ge 0} \cup \{\infty\}$. 

Now we define the convergence of spectral structures. 
In the below, when we talk about a net of spectral structure $\{\Sigma_\alpha\}_\alpha = \{(H_\alpha, A_\alpha)\}_\alpha$, we always assume that $\{H_\alpha\}_\alpha$ is a convergent net of Hilbert spaces. 
\begin{dfn}[{\cite[Theorem 2.4 and Definition 2.14]{KuwaeShioya2003}}]\label{def_spec_conv.y}
\normalfont
Given a net of spectral structures $\{\Sigma_\alpha\}_{\alpha \in \mathcal{A}} =\{(H_\alpha, A_\alpha)\}_{\alpha \in \mathcal{A}}$, we say that $\{\Sigma_\alpha\}_\alpha ${\it strongly} (resp. {\it compactly}){\it converges to} $\Sigma_\infty$ if $E_\alpha((\lambda, \mu]) \to E_\infty((\lambda, \mu])$ strongly (resp. compactly) for any real numbers $\lambda < \mu$ which are not in the point spectrum of $A_\infty$. 
\end{dfn}

In terms of the spectrum of $A_\alpha$, the followings hold.  
\begin{fact}[{\cite[Proposition 2.6 and Remark 2.8]{KuwaeShioya2003}}]
Let $a < b$ be two numbers which are not in the point spectrum of $A_\infty$. 
If $\Sigma_\alpha \to \Sigma_\infty$ strongly, we have 
\[
\liminf_\alpha n_\alpha((a, b]) \ge n_\infty((a , b]). 
\]
\end{fact}

\begin{fact}[{\cite[Theorem 2.6 and Remark 2.8]{KuwaeShioya2003}}]
Assume that $\Sigma_\alpha \to \Sigma_\infty$ compactly. 
Then for any $a, b \in \R \setminus \sigma(A_\infty)$ with $a < b$, we have
$n_\alpha((a, b]) = n_\infty((a , b])$
for $\alpha$ sufficiently close to $\infty$. 
In particular, the limit set of $\sigma(A_\alpha)$ coincides with $\sigma(A_\infty)$. 
\end{fact}

Next, we focus on the case of positive spectral structures. 
If $(H, A)$ is a positive spectral structure, its associated quadratic form $\mathcal{E} \colon H\to [0, \infty]$ is defined by $\mathcal{E}(u)
:=\|\sqrt{A}u\|^2_{H}$ 
for $u\in \mathcal{D}(\sqrt{A})$ and $\mathcal{E}(u):=\infty$ 
for $u\in H\backslash\mathcal{D}(\sqrt{A})$. 
Since $A$ is a closed operator, we see that $\mathcal{E}$ is 
{\it closed}, namely, 
$\mathcal{D}(\sqrt{A})$ is complete 
with respect to the norm 
defined by $\|u\|_{\mathcal{E}}:=\sqrt{\| u\|_H^2+\mathcal{E}(u)}$. 
We also have a notion of convergence for quadratic forms, as follows. 

\begin{dfn}[{\cite[Definition 2.11 and 2.13]{KuwaeShioya2003}}]
\normalfont
Let $\{H_\alpha\}_{\alpha \in \mathcal{A}}$ be a convergent net of Hilbert spaces. 
A net of closed quadratic forms 
$\{ \mathcal{E}_\alpha\colon H_\alpha \to [0, \infty]\}_\alpha$ 
{\it Mosco converges to 
$\mathcal{E}_\infty\colon H_\infty \to [0, \infty]$ 
} as $\alpha\to \infty$ 
if 
\begin{itemize}
\setlength{\parskip}{0cm}
\setlength{\itemsep}{0cm}
 \item[(1)] $\mathcal{E}_\infty(u_\infty)
\le \liminf_{\alpha\to\infty}\mathcal{E}_\alpha(u_\alpha)$ for any $\{ u_\alpha\}_\alpha$
with $u_\alpha\to u_\infty$ weakly, and
 \item[(2)] for any $u_\infty\in H_\infty$ there exists $\{ u_\alpha\}_\alpha$ 
strongly converging to $u_\infty$ such that 
$\mathcal{E}_\infty(u_\infty) = \lim_{\alpha\to\infty}\mathcal{E}_\alpha(u_\alpha)$. 
\end{itemize}
Moreover, 
$\{ \mathcal{E}_\alpha\}_\alpha$ 
{\it compactly converges to 
$\mathcal{E}_\infty$ 
} as $\alpha\to \infty$ 
if 
\begin{itemize}
\setlength{\parskip}{0cm}
\setlength{\itemsep}{0cm}
 \item[(3)] $\{ \mathcal{E}_\alpha\}_\alpha$ 
Mosco converges to 
$\mathcal{E}_\infty$ as $\alpha\to \infty$, and
 \item[(4)] for any $\{ u_\alpha\}_\alpha$ with 
$\limsup_{\alpha\to\infty}(\| u_\alpha\|_{H_\alpha}^2 + \mathcal{E}_\alpha(u_\alpha)) < \infty$, 
there exists a strongly convergent subnet. 
\end{itemize}
\label{def spec2}
\end{dfn}

The spectral convergences of positive spectral structures have equivalent definitions in terms of convergence of associated quadratic forms, as follows. 

\begin{fact}[{\cite[Theorem 2.4]{KuwaeShioya2003}}]\label{fact_KS_conv.h}
Given a net of positive spectral structures $\{\Sigma_\alpha\}_\alpha = \{(H_\alpha, A_\alpha)\}_{\alpha}$ let us denote the corresponding net of quadratic forms by $\{\mathcal{E}_\alpha\}_\alpha$. 
Then the followings are equivalent. 
\begin{enumerate}
    \item We have a Mosco convergence $\mathcal{E}_\alpha \to \mathcal{E}_\infty$ (resp. $\mathcal{E}_\alpha \to \mathcal{E}_\infty$ compactly). 
    \item $\{\Sigma_\alpha\}_\alpha $ strongly (resp.  compactly) converges to $\Sigma_\infty$
\end{enumerate}
\end{fact}

Note that when $\mathcal{A} = \R_{\ge 0}$ with $0\in \R_{\ge 0}$ regarded as the limit element $\infty \in \mathcal{A}$, we see that any convergence of a net $\{X_s\}_{s > 0}$ is equivalent to the convergence of subsequence $\{X_{s_i}\}_{i \in \Z_{> 0}}$ for all $\{s_i\}_{i \in \Z_{> 0}}$ with $\lim_{i \to \infty} s_i = 0$. 
Thus in the below, we mainly work in the case where $\mathcal{A} = \Z_{> 0} \sqcup \{\infty\}$, i.e., we work with sequences.

\subsection{Lie group actions on Spectral structures}\label{subsec_action_spec.h}
Let $\Sigma$ be a spectral 
structure on $H$ whose infinitesimal 
generator is 
$A\colon \mathcal{D}(A)\to H$ 
and 
$G$ be a compact Lie group. 
Suppose that $G$ acts on 
$H$ linearly and isometrically, and $G\cdot \mathcal{D}(A)\subset \mathcal{D}(A)$ and suppose that 
$A$ is $G$-equivariant. 
For a finite dimensional 
unitary representation 
$(\rho, V)$ of $G$, 
we define the spectral 
structure $\Sigma^\rho$ on 
\begin{align*}
H^\rho&:=(H\otimes V)^\rho\\
&= \left\{ \sum_i u_i\otimes v_i\in H\otimes V;\, 
\sum_i (\gamma\cdot u_i)\otimes \rho(\gamma)v_i = \sum_i u_i\otimes v_i\right\}
\end{align*}
as follows. 
Since 
\[
A\otimes{\rm id}_V \colon 
\mathcal{D}(A)\otimes V\to H\otimes V, 
\]
is $G$-equiavariant, 
we obtain the map 
\[
A^\rho:=(A\otimes{\rm id}_V)|_{(\mathcal{D}(A)\otimes V)^\rho}
\colon (\mathcal{D}(A)\otimes V)^\rho\to (H\otimes V)^\rho.
\]
Then we have the spectral 
structure $\Sigma^\rho$ whose 
infinitesimal generator is $A^\rho$. 

Let $E$, $E^\rho$ be the spectral measures of $A$, $A^\rho$, respectively. 
Then one can see 
\[
E^\rho((\lambda,\mu])
= E((\lambda,\mu])\otimes {\rm id}_V
\colon H^\rho \to H^\rho. 
\]

Let $(H_\alpha,\Sigma_\alpha)$ 
be the net of spectral structures 
and $\{ H_\alpha\}_\alpha$ converge to 
$H_\infty$.  
Let $\Phi_\alpha\colon \mathcal{C}
\to H_\alpha$ be as in Definition 
\ref{def_conv_Hilb.h}. 
We suppose that 
$G$ acts linearly and isometrically 
on all of $H_\alpha$ and 
$A_\alpha$ are all $G$-equivariant. 
Moreover we also assume that 
$G\cdot \mathcal{C}\subset \mathcal{C}$ 
and $\Phi_\alpha$ are $G$-equivariant. 
Put 
\begin{align*}
\mathcal{C}^\rho&:= (\mathcal{C}\otimes V)^\rho,\\
\Phi_\alpha^\rho&:= \Phi_\alpha
\otimes{\rm id}_V|_{\mathcal{C}^\rho}
\colon \mathcal{C}^\rho\to H_\alpha^\rho,\\
\end{align*}
then we can see that $\{ H_\alpha^\rho\}_\alpha$ 
converges to $H_\infty^\rho$. 
One can show the following 
proposition. 
\begin{prop}
If $\Sigma_\alpha\to \Sigma_\infty$ 
strongly (resp.compactly), 
then $\Sigma_\alpha^\rho\to \Sigma_\infty^\rho$ 
strongly (resp.compactly).
\label{prop_spec_to_equiv_spec.h}
\end{prop}

\subsection{Strong spectral convergence of equivariant Laplacians}\label{subsec_spec_conv_eq_lap.h}

In this subsection, we explain how to apply the general theory of subsection \ref{subsec_spec_conv_general.y} to our situations.

The following notion 
is the special case of 
\cite[Definition 4.1]{FukayaYamaguchi1994}. 
\begin{dfn}
\normalfont 
Let $G$ be a compact Lie group.
\begin{itemize}
\setlength{\parskip}{0cm}
\setlength{\itemsep}{0cm}
 \item[(1)] 
Let $(P',d')$ and $(P,d)$ be metric spaces with 
isometric $G$-action. 
A map $\phi:P'\to P$ is an {\it $G$-equivariant 
$\varepsilon$-approximation} if $\phi$ is 
$G$-equivariant and $\varepsilon$-approximation. 
Here, $\varepsilon$-approximation means that 
$|d'(x',y') - d(\phi(x'),\phi(y'))| < \varepsilon$ 
holds for all $x',y'\in P'$ and 
$P\subset B(\phi(P'),\varepsilon)$. 
Moreover if $\phi$ is a Borel map then 
it is called a {\it Borel $G$-equivariant 
$\varepsilon$-approximation}. 
 \item[(2)] Let $\{(P_i,d_i,\nu_i,p_i)\}_i$ be a sequence of pointed 
metric measure spaces with isometric $G$-action. 
$(P_\infty,d_\infty,\nu_\infty,p_\infty)$ is said to be 
{\it the pointed $G$-equivariant measured Gromov-Hausdorff limit of 
$\{(P_i,d_i,\nu_i,p_i)\}_i$}, 
or 
\begin{align*}
(P_i,d_i,\nu_i,p_i)
\xrightarrow{G\mathchar`-\mathrm{pmGH}}
(P_\infty,d_\infty,
\nu_\infty,p_\infty), 
\end{align*}
if $G$ acts on $P_\infty$ isometrically and there are positive numbers 
$\{ \varepsilon_i\}_i$, $\{ R_i\}_i$, 
$\{ R_i'\}_i$ with 
\begin{align*}
\lim_{i\to \infty}\varepsilon_i = 0,\quad \lim_{i\to \infty}R_i=\lim_{i\to \infty}R_i'=\infty,
\end{align*}
and Borel $G$-equivariant 
$\varepsilon_i$-approximation 
\begin{align*}
\phi_i\colon (\pi_i^{-1}(B(\bar{p}_i,R_i')),p_i)\to (\pi_\infty^{-1}(B(\bar{p}_\infty,R_i)),p_\infty)
\end{align*}
for every $i$ 
such that 
\begin{align*}
\limsup_{i\to\infty}\left| 
\int_{P_\infty}f d\nu_\infty
- \int_{P_i}f\circ\phi_i d\nu_i
\right| = 0
\end{align*}
for any $f\in C_c(P_\infty)$. 
Here, $\pi_i\colon P_i\to P_i/G$ is the quotient map and 
$\bar{p}_i=\pi_i(p_i)$. 
\end{itemize}
\label{def GmGH.h}
\end{dfn}

\begin{thm}\label{fact_strong_conv.y}
Let $G$ be a compact Lie group. 
Let $n \in \Z_{>0}$ and $\kappa \in \R$. 
Assume that we have a family of pointed Riemannian manifolds $\{(P_i, g_i, p_i)\}_{i \in \Z_{>0}}$ with isometric $G$-actions, and each of them satisfies the condition
\[
\dim P_i = n \mbox{ and } \mathrm{Ric}(g_i) \ge \kappa g_i.  
\]
Assume there exists a pointed metric measure space $(P_\infty, d_\infty, \nu_\infty, p_\infty)$ with an isometric measure-preserving $G$-action, and we have 
\[
(P_i, d_i, \nu_i, p_i) \xrightarrow{G\mathchar`-\mathrm{pmGH}} (P_\infty, d_\infty, \nu_\infty, p_\infty). 
\]
Here $(P_i,d_i, \nu_i )$ is the Riemannian manifold $(P_i, g_i)$ regarded as a metric measure space. 
Let $(\rho, V)$ be a finite dimensional representation of $G$. 
Set $H_i^\rho := (L^2(P_i, \nu_i) \otimes V)^{\rho}$ and $A_i^\rho := \Delta_i\otimes{\rm id}_V|_{H_i^\rho}$ and let $\Sigma_i^\rho$ be the spectral structure induced by $A_i^\rho$ for each $i \in \Z_{>0} \cup \{\infty\}$. 
Then we have $\Sigma_i^\rho \to \Sigma_\infty^\rho$ strongly. 
\label{thm_equiv_kuwaeshioya.h}
\end{thm}
\begin{proof}
Put $H_i^\rho := L^2(P_i, \nu_i)$, 
$A_i:=\Delta_i$ and 
let $\Sigma_i$ be the spectral structure induced by $A_i$ for each $i \in \Z_{>0} \cup \{\infty\}$. 
Then by \cite[Theorem 1.3]{KuwaeShioya2003}, 
one can see that 
$\Sigma_i \to \Sigma_\infty$ strongly. Recall that 
\begin{align*}
\Phi_i\colon \mathcal{C}:=
C_c(P_\infty) \to H_i
\end{align*}
was defined by 
\[ 
\Phi_i(f)(u):=
\left\{ 
\begin{array}{cc}
f\circ\phi_i(u) & u\in \pi_i^{-1}(B(\bar{p}_i,R'_i))\\
0 & u\notin \pi_i^{-1}(B(\bar{p}_i,R'_i)) 
\end{array}
\right.
\]
for $f\in\mathcal{C}$. 
Since $\phi_i$ is $G$-equivariant, 
$\Phi_i$ is also $G$-equivariant. 
Then by Proposition \ref{prop_spec_to_equiv_spec.h}, 
$\Sigma_i^\rho \to \Sigma_\infty^\rho$ strongly. 
\end{proof}

Next we consider the following 
situation. 
We have a family of closed Riemannian manifolds $\{(P_i, g_i)\}_{i \in \N}$ with isometric $G$-actions 
and 
suppose that the dimension 
of $P_i$ is independent of 
$i$. 
We have a fixed positive integer $N >0$, and points $p_i^j\in P_i$ for each $i \in \N$ and $1 \le j \le N$. 
We also assume that for each $j \neq l$, we have $\lim_{i \to \infty} d_i(p_i^j, p_i^l) = \infty$. 
We assume that for each $1\le j\le N$, there exists a pointed metric measure space $(P_\infty^j, d_\infty^j, \nu_\infty^j, p_\infty^j)$ with isometric measure-preserving 
$G$-action such that 
\begin{equation}\label{eq_conv_hoge.y}
(P_i, d_i, \nu_i, p_i^j) \xrightarrow{G\mathchar`-\mathrm{pmGH}} (P_\infty^j, d_\infty^j, \nu_\infty^j, p_\infty^j). 
\end{equation}
Here $(P_i,d_i, \nu_i )$ is the Riemannian manifold $(P_i, g_i)$ regarded as a metric measure space. 
We also assume that 
there is $\kappa\in\R$ 
such that 
${\rm Ric}_{g_i}
\ge \kappa g_i$, 
then 
the Laplacian $\Delta_{\infty}^j$ acting on $L^2(P_\infty^j, \nu_\infty^j)$ makes sense 
by \cite{Cheeger-Colding3}.

Fix a positive integer $k \ge 1$. 
The Hilbert spaces we consider are 
\begin{align*}
H_i &:= L^2(P_i, \nu_i) , \\
H_\infty &:= \oplus_{j = 1}^{N} L^2(P_\infty^j, \nu_\infty^j). 
\end{align*}
Then we obtain $H_i^\rho$ 
and $H_\infty^\rho$ in the same way 
as Subsection \ref{subsec_action_spec.h}. 
Now we explain the natural choice of $\mathcal{C}$ and $\Phi_i$. 
In the case of $N = 1$, 
put $\mathcal{C}$ and $\Phi_i$ 
as in the proof of 
Theorem \ref{thm_equiv_kuwaeshioya.h}.

If $N\ge 2$, we can modify the above constructions as follows. 
By the convergence (\ref{eq_conv_hoge.y}), we can choose positive numbers $\epsilon_i, R'_i, R_i$ such that
$\lim_{i\to\infty}\varepsilon_i=0$ 
and $\lim_{i\to\infty}R_i = \lim_{i\to\infty}R'_i
=\infty$ and 
$G$-equivariant 
$\varepsilon_i$-approximation 
\begin{align*}
\phi_i^j\colon \pi_i^{-1}(B(\bar{p}^j_i,R'_i))\to 
\pi_\infty^{-1}(B(\bar{p}^j_\infty,R_i))
\end{align*}
such that $\phi_i(p_i) = p_\infty$. 
Moreover, by the assumption that
$\lim_{i \to \infty} d_i(p_i^j, p_i^l) = \infty$ for $j \neq l$, we may assume that for each $i$, the sets $\{\pi^{-1}_i(B(\bar{p}^j_i,R'_i))\}_{j=1}^N$ are mutually disjoint. 
Thus we can set
\begin{align*}
\mathcal{C}:=\oplus_{j=1}^N C_c(P^j_\infty)
=\left\{ \sum_{j=1}^N f_j\in \oplus_{j = 1}^N C(P^j_\infty);\, {\rm supp}(f_j)\mbox{ is compact.}\right\},
\end{align*}
\[ 
\Phi_i(f)(u):=
\left\{ 
\begin{array}{cc}
f\circ\phi_i^j(u) & u\in \pi_i^{-1}(B(\bar{p}^j_i,R'_i))\\
0 & u\notin \pi_i^{-1}(B(\bar{p}^j_i,R'_i)) \mbox{ for any }j
\end{array}
\right.
\]
for $f\in \mathcal{C}$. 
Then the same procedure 
in Subsection \ref{subsec_action_spec.h} yields 
$\mathcal{C}^\rho$ and 
$\Phi_i^\rho$. 

Set
$A_i := \Delta_i$ and $A_\infty := \oplus_{j=1}^N \Delta_\infty^j$. 
Then we obtain $\Sigma_i^\rho$ 
and $\Sigma_\infty^\rho$ in 
the same way as Subsection \ref{subsec_action_spec.h}. 

Now we show that, under the lower bound of Ricci curvature of $\{P_i\}_{i \in \Z_{\ge 0}}$, we have the strong convergence $\Sigma_i^\rho \to \Sigma_\infty^\rho$ as follows. 
\begin{prop}\label{prop_strong_conv.y}
Under the convergence 
\eqref{eq_conv_hoge.y}, 
assume moreover that
there exist $n \in \Z_{>0}$ and $\kappa >0$ such that for all $i \in \Z_{>0}$, we have
\[
\dim P_i = n \mbox{ and } \mathrm{Ric}(g_i) \ge \kappa g_i. 
\]
Then we have $\Sigma_i^\rho \to \Sigma_\infty^\rho$ strongly.
\end{prop}

\begin{proof}
Take any two real numbers $\lambda< \mu$ which are not in the point spectrum of $A_\infty^\rho$.
Then we must show that $E_i^\rho((\lambda, \mu]) \to E_\infty^\rho((\lambda, \mu])$ strongly. 
To simplify notations, we write $E_i := E_i^\rho((\lambda, \mu])$ in this proof. 
Take a strongly convergent sequence $u_i \to u_\infty$, where $u_i \in H_i^\rho$. 
We must show $E_i u_i \to E_\infty u_\infty$ strongly. 

We write $H_\infty^{\rho,j} :=(L^2(P_\infty^j, \nu_\infty^j) \otimes V)^{\rho}$ so that $H_\infty^\rho = \oplus_{j = 1}^N H_\infty^{\rho,j}$. 
The spectral structure decompose accordingly, and we write corresponding objects for each component as $\Sigma_\infty^{\rho,j}$ and $E_\infty^{\rho,j}$. 
We decompose $u_\infty = \sum_{j = 1}^N u_\infty^j$ where $u_\infty^j \in H_\infty^{\rho,j}$. 

We may decompose the sequence $\{u_i\}_i$ into sequences $\{u_i^j\}_i$ ($1 \le j \le N $), where $u_i = \sum_{j = 1}^N u_i^j$ and $u_i^j \to u_\infty^j$ strongly for each $j$. 
By the lower bound for Ricci curvature, 
we can apply Theorem  
\ref{fact_strong_conv.y} and we know that ${\Sigma}_i^\rho \to {\Sigma}_\infty^{\rho,j}$ strongly as $i \to \infty$.
Thus we have ${E}_iu_i^j \to E_\infty^j u_\infty^j$ strongly. 
We take a sum over $j$ and get the result. 
\end{proof}

\subsection{Ricci curvature}\label{subsec_ric.h}
In this subsection 
let $(X,g)$ be a Riemannian manifold 
and $\pi\colon S\to X$ 
be a principal $S^1$-bundle. 
Suppose that an $S^1$ connection 
$\sqrt{-1}A\in \Omega^1(S,\sqrt{-1}\R)$ is given. 
We define a Riemannian metric $\hat{g}$ 
on $S$ from $A$ and $g$ similarly as in Subsection \ref{sec cpx str.h}. 
Here, we compute the Ricci curvature 
of $\hat{g}$. 

Let $x^1,\ldots,x^N$ be a local coordinate 
of $X$ and denote by $\hat{\partial}_i$ 
the horizontal lift of $\frac{\del}{\del x^i}$. 
Denote by 
$\xi^\sharp\in C^\infty(X; TX)$ the vector field 
generated by $\xi\in {\rm Lie}(S^1)$. 
Put $e:=\sqrt{-1}\in {\rm Lie}(S^1)$, 
define $F_{ij}$ by 
\begin{align*}
F_{ij} e := F^A(\hat{\partial}_i,\hat{\partial}_j)
\end{align*}
and let $\Gamma_{ij}^k$ be the Christoffel 
symbols of $g$. 
Since $F^A$ is a basic $2$-form on $S$, 
and since $S^1$ is abelian, 
we can see that $F_{ij}$ is $S^1$-invariant 
and $F^A$ is the pullback of 
$\frac{1}{2}F_{ij}dx^i\wedge dx^j\in\Omega^2(X)$. 
In this situation we have 
\begin{align*}
[\hat{\partial}_i, \hat{\partial}_j]
&= -F_{ij} e^\sharp, \quad
\nabla_{e^\sharp} e^\sharp =
[e^\sharp, e^\sharp]=[\hat{\partial}_i, e^\sharp]
=0,\\
\nabla_{\hat{\partial}_i} \hat{\partial}_j 
&= \Gamma_{ij}^k \hat{\partial}_k 
-\frac{1}{2} F_{ij} e^\sharp,
\quad \nabla_{\hat{\partial}_i} e^\sharp
= \nabla_{e^\sharp} \hat{\partial}_i
= \frac{g^{kh}F_{ih} }{2}
\hat{\partial}_k.
\end{align*}
Put 
\begin{align*}
(\nabla F)_{kij}
= \hat{\partial}_k(F_{ij}) -F_{il}\Gamma_{jk}^l
- F_{lj}\Gamma_{ik}^l,
\end{align*}
then the 2nd Bianchi identity implies 
\begin{align*}
0 
= dF^A(\hat{\partial}_i, \hat{\partial}_j, \hat{\partial}_k)
&= \{ \hat{\partial}_i(F_{jk}) - \hat{\partial}_j(F_{ik}) +  \hat{\partial}_k(F_{ij})\}e\\
&=\{ (\nabla F)_{ijk} + (\nabla F)_{jki} 
+ (\nabla F)_{kij}\} = 0.
\end{align*}

Now we denote by $\hat{R}$ the 
curvature tensor of $\hat{g}$, and by $R$ that of $g$. 
Then we have
\begin{align*}
\hat{R}(\hat{\partial}_i, \hat{\partial}_j) \hat{\partial}_k
&= R_{ijk}^l \hat{\partial}_l
+\frac{(\nabla F)_{kij}}{2}e^\sharp \\
&\quad\quad +
\frac{g^{lh}}{4}
(2F_{ij} F_{kh} - F_{jk} F_{ih} - F_{ki} F_{jh})\hat{\partial}_l,\\
\hat{R}(\hat{\partial}_i, e^\sharp) \hat{\partial}_j
&= \frac{g^{lh}}{2}
(\nabla F)_{ijh}\hat{\partial}_l
-\frac{g^{kh}F_{jh} F_{ik}e^\sharp}{4},\\
\hat{R}(\hat{\partial}_i, e^\sharp) 
e^\sharp
&= -\frac{g^{kh} F_{ih}g^{lp}F_{kp}}{4}\hat{\partial}_l,\\
\hat{R}(e^\sharp, e^\sharp) 
e^\sharp
&= 0.
\end{align*}

Now, define $F^*F\in\Gamma({\rm Symm}_2(H^*))\otimes 
{\rm Symm}_2(\mathfrak{g})$ by 
\begin{align*}
F^*F = g^{kl}F_{ik} F_{jl} \hat{\partial}^i
\otimes \hat{\partial}^j \otimes e
\otimes e
\end{align*}
where $H^*$ is the dual bundle of the horizontal distribution 
$H\subset TP$ and $\{ \hat{\partial}^i\}_i$ is the dual basis of 
$\{ \hat{\partial}_i\}_i$. 
Note that $\{ (d^\nabla)^* F\}_j = -g^{ih} (\nabla F)_{hij}$. 
Then we have 
\begin{align*}
\hat{{\rm Ric}}(\hat{\partial}_j,\hat{\partial}_k)
&= {\rm Ric}_{jk} 
-\frac{(F^*F)_{jk}}{2},\\
\hat{{\rm Ric}}(\hat{\partial}_j,e^\sharp)
&= \frac{ \{ (d^\nabla)^* F\}_j}{2},\\
\hat{{\rm Ric}}(e^\sharp, e^\sharp)
&= 
\frac{g^{jk} (F^*F)_{jk}}{4}.
\end{align*}
Here, $(d^\nabla)^* F$ is given by the 
pullback of $d^*(F_{ij}dx^i dx^j)$. 
\begin{prop}
Let $(X^{2n},\omega)$ be a symplectic manifold 
and $(L,h,\nabla)$ is the prequantum line bundle. 
In the above setting, if $S=S(L,h)$, 
$g=g_J$ for some $\omega$-compatible 
almost complex structure and 
$A$ is the $S^1$-connection on $S$ 
corresponding to $\nabla$, then 
\begin{align*}
\hat{{\rm Ric}}(\hat{\partial}_j,\hat{\partial}_k)
= {\rm Ric}_{jk} 
-\frac{g_{jk}}{2},\quad
\hat{{\rm Ric}}(\hat{\partial}_j,e^\sharp)
= 0,\quad
\hat{{\rm Ric}}(e^\sharp, e^\sharp)
= 
\frac{n}{2}.
\end{align*}
\label{prop_Ric_of_S.h}
\end{prop}
\begin{proof}
By the assumption 
$F=-\sqrt{-1}\pi^*\omega$ holds. 
Then we have $F^*F = g$, 
hence 
\begin{align*}
\hat{{\rm Ric}}(\hat{\partial}_j,\hat{\partial}_k)
&= {\rm Ric}_{jk} 
-\frac{g_{jk}}{2},\\
\hat{{\rm Ric}}(e^\sharp, e^\sharp)
&= 
\frac{n}{2}.
\end{align*}
To show 
$(d^\nabla)^* F=0$, 
it suffices to show $d^*\omega=0$. 
Since $\omega=g(J\cdot,\cdot)$ holds 
and $g$ is hermitian with respect to $J$, 
$*\omega= c\omega^{n-1}$ holds 
for some constant $c$. 
Since $d*\omega=cd(\omega^{n-1})=0$, 
we have the assertion. 
\end{proof}

\section{The compact spectral convergence}\label{sec localization.h}
In this section, we prove our first main theorem of this paper, Theorem \ref{thm_main_intro.y}. 
By the identifications of spectral structures given by \eqref{eq_delbarlap.y} and \eqref{eq_gaussian=eqlap.y}, this is equivalent to Theorem \ref{thm_main.y}. 
Since we know the strong convergence by Proposition \ref{prop_strong_conv.y}, in order to show the compact convergence, what we need to show is the item (4) of Definition \ref{def spec2}, i.e., that given any sequence $\{f_i \in (L^2(S; \hat{g}_{J_{s_i}})\otimes \C)^{\rho_k}\}_i$ with $\limsup_{i\to\infty}\left(
\| f_i\|_{L^2}^2 + \| df_i\|_{L^2}^2
\right) <\infty$, we can find a strongly convergent subsequence. 
In order for this, what we need to prove is, roughly speaking, that given any such sequence $\{f_i\}_i$, they stay in a certain distance from the set $B_k$ of Bohr-Sommerfeld points of level $k$. 

In subsection \ref{subsec_loc_est.y}, as a preparation for the localization argument, we show a local estimate of the lower bound for the laplacian $\Delta_{\hat{g}_J}^{\rho_k}$ with Dirichlet boundary conditions (Proposition \ref{prop_loc_est.y}). 
Using this, in subsection \ref{subsec_loc_bs.y}, we show the localization of $H^{1,2}$-bounded sequence to the set of Bohr-Sommerfeld points of level $k$ (Proposition \ref{thm_loc_bs. y}).
Combining this result with the lower boundedness of Ricci curvatures, in subsection \ref{sec. conv.h}, we prove Theorem \ref{thm_main.y}.  

\subsection{A local estimate}\label{subsec_loc_est.y}
\begin{lem}
Let $\hat{g}$ be an inner product on
a finite dimensional vector space $T$ and 
$W\subset T$ be a subspace. 
we have 
\begin{align*}
\hat{g}^{-1}( \alpha, \alpha) 
\ge (\hat{g}|_W)^{-1}( \alpha|_W, \alpha|_W) 
\end{align*}
for any $\alpha\in T^*$. 
\label{linear alg.h}
\end{lem}
\begin{proof}
Along the orthogonal decomposition 
$T= W\oplus W^\perp$, 
we decompose $\hat{g}$ into 
\[ \hat{g}=
\left (
\begin{array}{cc}
\hat{g}|_W & 0 \\
0 & \hat{g}|_{W^\perp} 
\end{array}
\right ).
\]
Then it induces the orthogonal decomposition 
$T^*\cong W^*\oplus (W^\perp)^*$ 
and 
\[ \hat{g}^{-1}=
\left (
\begin{array}{cc}
(\hat{g}|_W)^{-1} & 0 \\
0 & (\hat{g}|_{W^\perp})^{-1} 
\end{array}
\right ),
\]
which gives the assertion. 
\end{proof}

Let $(X^{2n}, \omega)$ be a closed symplectic manifold with a prequantum line bundle $(L, \nabla, h)$, $\mu \colon (X, \omega) \to B$ be a possibly singular Lagrangian fibration, and $J$ be an $\omega$-compatible almost complex structure. 
We denote the frame bundle of $L$ by $\pi \colon S \to X$. 
Let $V\subset B$ be an open subset on which $\mu$ is non-singular with connected torus fibers, equipped with a fixed action-angle coordinate on $U := \mu^{-1}(V)$. 
For each $b \in V$,  
put $X_b:=\mu^{-1}(b)$ and 
$S_b:=\pi^{-1}(X_b)$, and 
denote by $\hat{g}_b$ the metric 
on $S_b$ induced by $\hat{g}_J$. 

Now denote the action-angle coordinate 
$x_1,\ldots,x_n,\theta^1,\ldots,\theta^n$ 
on $U$ such that 
$x_1,\ldots,x_n$ is a coordinate on $V$ 
and $\nabla=d -\sqrt{-1} x_id\theta^i$. 
Put $b_i:=x_i(b)$. 
Then one can see that 
$\mu^{-1}(b)$ is a Bohr-Sommerfeld fiber of level $k$ 
iff $( b_1,\ldots,b_n)\in \frac{1}{k}\Z$. 

Now fix $b \in V$, put $g_b=g_{ij}d\theta^i d\theta^j$ and 
\begin{align}\label{eq_loc_est_1.y}
N_b
&:= \sup_{\theta \in T^n}\left\{ N_b(\theta)\in\R_{+};\, 
N_b(\theta) \mbox{ is the maximum eigenvalue of }
(g_{ij}(\theta))_{i,j}\right\},\\
\lambda(k,b)
&:= \inf \left\{ \sum_{i=1}^n ( m_i+kb_i)^2;\, m_1,\ldots,m_n\in\Z \right\}.
\end{align}
Here, $N_b$ and $\lambda(k,b)$ 
may depend on the choice of 
the action-angle coordinates. 

\begin{prop}\label{prop_loc_est_fiberwise.y}
For any $\varphi\in C^\infty(T^n,\C)$, 
\begin{align*}
\int_{T^n} 
\left(\frac{\del \varphi}{\del \theta^i}
+\sqrt{-1} k b_i\varphi \right)
\left(\frac{\del \bar{\varphi}}{\del \theta^j}
-\sqrt{-1} k b_j\bar{\varphi} \right)g^{ij}
d\theta 
\ge \frac{\lambda(k,b)}{N_b}
\int_{T^n} |\varphi|^2 d\theta
\end{align*}
holds, where $d\theta = d\theta^1\cdots d\theta^n$. 
\end{prop}
\begin{proof}
Since 
\begin{align*}
&\quad\quad \int_{T^n} 
\left(\frac{\del \varphi}{\del \theta^i}
+\sqrt{-1} k b_i\varphi \right)
\left(\frac{\del \bar{\varphi}}{\del \theta^j}
-\sqrt{-1} k b_j\bar{\varphi} \right)g^{ij}
d\theta \\
&\ge \frac{1}{N_b}\int_{T^n} 
\left(\frac{\del \varphi}{\del \theta^i}
+\sqrt{-1} k b_i\varphi \right)
\left(\frac{\del \bar{\varphi}}{\del \theta^j}
-\sqrt{-1} k b_j\bar{\varphi} \right)\delta^{ij}
d\theta \\
&= \frac{1}{N_b}\int_{T^n} 
\delta^{ij}\left(-\frac{\del^2 \varphi}{\del \theta^i\del \theta^j}
-2\sqrt{-1}kb_j\frac{\del \varphi}{\del \theta^i}
+k^2 b_i b_j \varphi \right)
\bar{\varphi}
d\theta, 
\end{align*}
it suffices to evaluate the lowest eigenvalue of 
the operator 
\begin{align*}
L_k:=\delta^{ij}\left(-\frac{\del^2 }{\del \theta^i\del \theta^j}
-2\sqrt{-1}kb_j\frac{\del }{\del \theta^i}\right)
+k^2 \| b\|^2.
\end{align*}
If we put $\varphi_m(\theta)=e^{\sqrt{-1}m_i\theta^i}$ 
for $m=(m_1,\ldots, m_n)\in\Z$, then 
\begin{align*}
L_k\varphi_m
&= \left( \| m\|^2+2km\cdot b +k^2\| b\|^2
\right)\varphi_m \\
&= \| m+kb\|^2\varphi_m,
\end{align*}
which gives the assertion. 
\end{proof}

\begin{prop}\label{prop_loc_est.y}
Let 
$b\in V,N_b$ and 
$\lambda(k,b)$ be as above 
and put 
$K = \inf_{b \in V}\frac{\lambda(k, b)}{N_b}$. 
We have 
\begin{align*}
\int_{S|_{U}} |df|_{\hat{g}}^2 d\mu_{\hat{g}}
&\ge 2\pi (k^2 + K)\int_{S|_{U}} |f|^2 d\mu_{\hat{g}}
\end{align*}
for all $f\in (C^\infty(S)\otimes \C)^{\rho_k}$
\end{prop}
\begin{proof}
Let $f\in (C^\infty(S)\otimes \C)^{\rho_k}$. 
By Lemma \ref{linear alg.h}, 
\begin{align*}
\int_{S|_{U}} |df|_{\hat{g}}^2 d\mu_{\hat{g}}
&\ge \int_{S|_{U}} |df|_{S_b}|_{\hat{g}_b}^2 d\mu_{\hat{g}}
\end{align*}
holds. 
We have 
\begin{align*}
\hat{g}_b &= (dt - x_id\theta^i)^2
+ g_b,\\
d\mu_{\hat{g}} &= 
dt\cdot\omega^n= dtdx 
d\theta
\end{align*}
where $g_b=g_J|_{X_b}$. 
On $S_b$, we may write 
$f|_{S_b} = e^{-\sqrt{-1}kt}\varphi(\theta)$ 
for some $\varphi\in C^\infty(T^n; \C)$. 
Then 
\begin{align*}
\int_{S|_{U}} |df|_{S_b}|_{\hat{g}_b}^2 d\mu_{\hat{g}}
&= \int_V \left( \int_{S^1\times X_b} 
|df|_{S_b}|_{\hat{g}_b}^2 dt
d\theta\right)
dx.
\end{align*}
Since 
\begin{align*}
|df|_{S_b}|_{\hat{g}_b}^2 
&= k^2|\varphi|^2
+ \left(\frac{\del \varphi}{\del \theta^i}
+\sqrt{-1} k x_i\varphi \right)
\left(\frac{\del \bar{\varphi}}{\del \theta^j}
-\sqrt{-1} k x_j\bar{\varphi} \right)g_b^{ij},
\end{align*}
one can see 
\begin{align*}
\int_{S^1\times X_b} 
|df|_{S_b}|_{\hat{g}_b}^2 dt
d\theta
&= 2\pi\int_{T^n} 
\left\{ 
k^2|\varphi|^2
+ \left(\frac{\del \varphi}{\del \theta^i}
+\sqrt{-1} k x_i\varphi \right)
\left(\frac{\del \bar{\varphi}}{\del \theta^j}
-\sqrt{-1} k x_j\bar{\varphi} \right)g_b^{ij}
\right\}
d\theta.
\end{align*}
By Proposition \ref{prop_loc_est_fiberwise.y}, we obtain 
\begin{align*}
\int_{S|_{U}} |df|_{\hat{g}}^2 d\mu_{\hat{g}}
&\ge 2\pi(k^2 + K)\int_{S|_{U}} |f|^2 d\mu_{\hat{g}}.
\end{align*}
\end{proof}

\subsection{Localization of $H^{1,2}$-bounded functions to Bohr-Sommerfeld fibers}\label{subsec_loc_bs.y}

Suppose we are given a closed symplectic manifold $(X, \omega)$ and a prequantum line bundle $(L, \nabla, h)$ as in Section \ref{sec Setting.h}. 
Suppose also that we have a nonsingular Lagrangian fibration $\mu \colon X \to B$. 
We consider an asymptotically semiflat family of $\omega$-compatible almost complex structures $\{J_s\}_{0<s <\delta}$. 
Put $g_s = g_{J_s}$ and 
$\hat{g}_s = \hat{g}_{J_s}$. 
Recall that we have given a local description of these metrics in subsection \ref{sec local.h}. 

Let us denote $B_k \subset B$ the set of Bohr-Sommerfeld points of level $k$. 
In this subsection, using the local estimate in the last subsection, we show the following. 
\begin{prop}\label{thm_loc_bs. y}
Under the above settings, assume that for each $0<s < \delta$, a function $f_s \in (C^\infty(S) \otimes \C)^{\rho_k}$ is chosen so that $\|f_s\|_{L^2} = 1$ and $\sup_{0 < s < \delta } \|df_s\|_{L^2} < \infty$. 
Then for any $\epsilon > 0$, there exists $C > 0$ such that
for all $0 < s <\delta$, we have
\[
\|f_s|_{\mu^{-1}(B_s(B_k, C))} \|_{L^2}^2\geq 1 - \epsilon. 
\]
Here $B_s(B_k, C) = \{b \in B \ | \ \inf_{x \in B_k} d_{g_s}(\mu^{-1}(b), \mu^{-1}(x)) < C \}$. 
\end{prop}

\begin{proof}
Let $\Lambda := \sup_{0 < s < \delta} \|df_s\|^2_{L^2} $. 
Let us fix a finite open cover $\mathcal{V}$ of $B$ along with a fixed action-angle coordinate on $\mu^{-1}(V)$ for each $V\in \mathcal{V}$ so that $L$ is trivialized as $\nabla=d -\sqrt{-1} x_id\theta^i$. 

First we focus on one element $V \in \mathcal{V}$. 
Let us denote the action-angle coordinate on $U := \mu^{-1}(V)$ by 
$x_1,\ldots,x_n,\theta^1,\ldots,\theta^n$. 
By (\ref{eq_metric.y}) and \cite[Proposition 7.2]{hattori2019}, there exist positive constants $c_1, M > 0$ such that
\begin{align}
c_1 \| b_1 - b_2\| &\ge \sqrt{s}g_s(\mu^{-1}(b_1), \mu^{-1}(b_2)),  \label{eq_basemet.y} \\
N_b(s) &\leq sM \label{eq_fibermet.y}
\end{align}
holds for all $0<s < \delta$ and all $b, b_1, b_2 \in V$. 
Although the integrability of $J_s$ is assumed 
in \cite[Proposition 7.2]{hattori2019}, 
this assumption is not essential 
and we can 
obtain the same inequality without integrability. 
Here we denoted the Euclidean distance on $V$ given by the action-angle coordinate by $\|\cdot \|$, and $N_b(s)$ is the positive number defined in (\ref{eq_loc_est_1.y}) with respect to the metric $g_s$. 
Take $C_V := \sqrt{M\Lambda}c_1/(\sqrt{\epsilon}k)$. 
Then for any point $b \in V \setminus B_s(B_k, C_V)$ we have
$\lambda(k, b) \geq k^2(\sqrt{s}C_V/c_1)^2$. 
Thus we have
\[
\inf_{b \in V \setminus B_s(B_k, C_V)} \frac{\lambda(k, b)}{N_b(s)}
\geq \frac{k^2sC_V^2}{sMc_1^2} = \frac{\Lambda}{\epsilon}. 
\]
Take any open subset $V' \subset V$ and denote $U' := \mu^{-1}(V')$. 
Applying Proposition \ref{prop_loc_est.y} we have
\begin{align}\label{eq_proof_localization.y}
    \int_{S|_{U' \setminus \mu^{-1}(B_s(B_k, \tilde{C})) }} |df_s|_{\hat{g_s}}^2 d\mu_{\hat{g_s}}
&\ge 2\pi \left(k^2 + \frac{\Lambda}{\epsilon}\right)\int_{S|_{U' \setminus \mu^{-1}(B_s(B_k, \tilde{C})) }} |f_s|^2 d\mu_{\hat{g_s}}
\end{align}
for any $0<s<\delta$ and $\tilde{C} \geq C_V$. 

Next we work globally. 
We take $C := \max_{V \in \mathcal{V}}C_V $. 
Take any finite partition $\mathcal{W}$ of $X$ into manifolds with corners such that for each element $W \in \mathcal{W}$ there exist an element $V \in \mathcal{V}$ such that $W \subset V$. 
If we apply the inequality (\ref{eq_proof_localization.y}) on each $W \in \mathcal{W}$ and add them together, we get the inequality
\[
\int_{S|_{X \setminus \mu^{-1}(B_s(B_k, C)) }} |df_s|_{\hat{g_s}}^2 d\mu_{\hat{g_s}}
\ge 2\pi\left(k^2 + \frac{\Lambda}{\epsilon}\right)\int_{S|_{X \setminus \mu^{-1}(B_s(B_k, C)) }} |f_s|^2 d\mu_{\hat{g_s}}
\]
Since we have
\[
\int_{S|_{X \setminus \mu^{-1}(B_s(B_k, C)) }} |df_s|_{\hat{g_s}}^2 d\mu_{\hat{g_s}} 
\le \int_{S} |df_s|_{\hat{g_s}}^2 d\mu_{\hat{g_s}} 
\le \Lambda, 
\]
we get
\[
\| f_s|_{\mu^{-1}(B_s(B_k, C))}\|^2_{L^2} \geq 1 - \frac{\epsilon}{2\pi}. 
\]
This proves the proposition. 
\end{proof}

\begin{rem}
The above localization argument can be regarded as an analogue of {\it Witten deformation}, the argument originating from Witten's proof of Morse inequality \cite{witten1982}. 
In our situations, the fiberwise Laplacian of the Lagrangian fibration plays the role of the differential of a Morse function, which puts a potential term to the Laplacians.
This idea is essentially the one used by Furuta, Fujita and Yoshida in \cite{FFY2010}. 
There, they showed a localization result for indices of Dirac-type operator on fibrations, and the invertibility of fiberwise operators play the role of the potential term. 
The argument is more elementary in our situations, because we only have to consider the zeroth degree part of $\delb$-Laplacians. 
In particular, in contrast to their settings, we do not need to assume that the family of metrics $\{g_s\}_s$ are submersion metrics. 
\end{rem}

\subsection{Convergence of 
$H^{1,2}$-bounded sequences}\label{sec. conv.h}
In this section, we consider an asymptotically semiflat family $\{J_s\}_{0<s<\delta }$ of $\omega$-compatible almost complex structures.
Denote by 
$S_\infty^b:=\R^n\times S^1$ 
be the limit space appearing 
in Theorem \ref{mGH conv.h.} for each $b\in B_k$.  

Take $s_i>0$ such that  $\lim_{i\to \infty} s_i =0$. 
Put 
\begin{align*}
H_i&=\left(
L^2\left(
S,\frac{\nu_{\hat{g}_{s_i}}}{K\sqrt{s_i}^n}
\right)\otimes \C
\right)^{\rho_k},\\
H_\infty&=
\bigoplus_{b\in B_k}
\left( 
L^2(S_\infty^b,\nu_\infty)
\otimes \C
\right)^{\rho_k},\\
\mathcal{C} &=
\bigoplus_{b\in B_k}
\left( 
C_c(S_\infty^b,\nu_\infty)
\otimes \C
\right)^{\rho_k}.
\end{align*}

\begin{prop}\label{thm_compact_conv.y}
Let $s_i>0$ and $\lim_{i\to\infty }s_i = 0$. 
Take $f_i\in (C^\infty(S)\otimes \C)^{\rho_k}$ 
such that 
\[
\limsup_{i\to\infty}\left(
\| f_i\|_{L^2(
S,\frac{\nu_{\hat{g}_{s_i}}}{K\sqrt{s_i}^n}
)}^2 
+ \| df_i\|_{L^2(
S,\frac{\nu_{\hat{g}_{s_i}}}{K\sqrt{s_i}^n}
)}^2
\right) <\infty.
\]
Then there is a subsequence 
$\{ f_{i(j)}\}_{j=1}^\infty
\subset \{ f_i\}_{i=1}^\infty$ 
and $f_\infty^b\in (L^2(S_\infty^b)
\otimes\C)^{\rho_k}$ for 
every $b\in B_k$ such that 
$f_{i(j)}\to \left( 
f_\infty^b \right)_b$ as 
$j\to \infty$ 
strongly.
\end{prop}
\begin{proof}
Put 
\begin{align*}
\| f_i\|_{H^{1,2}}^2:=\| f_i\|_{L^2}^2
+\| df_i\|_{L^2}^2, 
\end{align*}
then we have 
$\sup_i \| f_i\|_{H^{1,2}} < \infty$. 
First of all we apply  
\cite[Theorem 4.9]{Honda2015} to this sequence. 
Denote by $B_i(u_b,R)$ and 
$B_\infty(u_\infty^b,R)$ the open ball 
with respect to $\hat{g}_{s_i}$ and 
$g_{k,\infty}$, respectively. 
Here, $u_\infty^b=(0,1)$ 
is the base point in  $S_\infty^b$. 
Then the $H^{1,2}$-norms of $f_i|_{B_i(u_b,R)}$ are bounded, 
accordingly \cite[Theorem 4.9]{Honda2015} implies 
that there is 
$f_\infty^{b,R}\in 
H^{1,2}(B_\infty(u_\infty^b,R))
\otimes\C$ 
and $f_{i(j)}|_{B_{i(j)}(u_b,R)}\to f_\infty^{b,R}$ strongly for 
some subsequence 
$\{f_{i(j)} \}_j\subset 
\{ f_i\}_i$. 

By taking subsequences 
inductively and by the 
diagonal argument, 
the subsequence can be taken 
such that the above convergence holds for 
any $R=1,2,3,\ldots$.

Therefore, we obtain 
\begin{align*}
f_\infty^{b,R}
&= f_\infty^{b,R'}|_{B_\infty(u_\infty^b,R)},
\end{align*}
for any $R'> R$.
Define $f_\infty^b\in L^2_{\mathrm{loc}}(S_\infty^b)\otimes\C$ by 
$f_\infty^b|_{B_\infty(u_\infty^b,R)}
=f_\infty^{b,R}$, then 
\begin{align*}
\sum_{b\in B_k}
\| f_\infty^b\|_{L^2}^2
= \sum_{b\in B_k}\lim_{R\to \infty}
\| f_\infty^{b,R}\|_{L^2}^2
&= \sum_{b\in B_k}\lim_{R\to \infty}
\lim_{j\to\infty}
\| f_{i(j)}|_{B_{i(j)}(u_b,R)} \|_{L^2}^2
\end{align*}
holds. 
Since $B_i(u_b,R)\cap 
B_i(u_{b'},R)$ 
is empty for $b\neq b'$ and 
sufficiently 
small $s_i$, then one can see 
\begin{align*}
\sum_{b\in B_k}\lim_{R\to \infty}
\lim_{j\to\infty}
\| f_{i(j)}|_{B_{i(j)}(u_b,R)} \|_{L^2}^2
\le \lim_{R\to \infty}
\lim_{j\to\infty}
\| f_{i(j)} \|_{L^2}^2
=\lim_{j\to\infty}
\| f_{i(j)} \|_{L^2}^2<\infty,
\end{align*}
accordingly, 
$f_\infty:=(f_\infty^b)_b \in \bigoplus_{b \in B_k} (L^2(S_\infty^b, \nu_\infty)\otimes \C)$. 

Next we show that 
$f_{i(j)}\to f_\infty$ strongly. 

By the strong convergence 
$f_{i(j)}|_{B_{i(j)}(u_b,R)}\to f_\infty^{b,R}$, 
we have 
\begin{align*}
\lim_{l\to\infty}\limsup_{j\to\infty}\left\| 
\Phi_{i(j)}(\tilde{f}^{l,b,R}) - f_{i(j)}|_{B_{i(j)}(u_b,R)} \right\|_{L^2}
=0
\end{align*} 
for some $\{ \tilde{f}^{l,b,R}\}_{l=0}^\infty \subset C_c(B_\infty(u_\infty^b,R))$ 
such that $\lim_{l\to \infty}\| \tilde{f}^{l,b,R} - f_\infty^{b,R}\|_{L^2}=0$. 
Then for any $R>0$ there is a sufficiently 
large integer $l_R>0$ such that 
\begin{align*}
\limsup_{j\to\infty}\left\| 
\Phi_{i(j)}(\tilde{f}^{l,b,R}) - f_{i(j)}|_{B_{i(j)}(u_b,R)} \right\|_{L^2}
<2^{-R}, \quad 
\left\| \tilde{f}^{l,b,R} - f_\infty^{b,R}\right\|_{L^2}<2^{-R}
\end{align*} 
holds for any $l\ge l_R$. Since $X$ is compact then $B_k$ is a finite set, 
consequently, we may take $l_R$ independently of $b$. 
Now put $\tilde{f}^{l,R}:=(\tilde{f}^{l,b,R})_b \in\mathcal{C}$. 

In order to show the strong convergence $f_{i(j)} \to f_\infty$, it is enough to show the followings. 
\begin{align}
  & \lim_{R \to \infty} \left\| \tilde{f}^{l_R,R} - f_\infty \right\|_{L^2} = 0, \label{eq_comp_conv_1.y} \\
  & \lim_{R\to \infty}\limsup_{j\to\infty}\left\| 
\Phi_{i(j)}(\tilde{f}^{l_{R},R}) - f_{i(j)} \right\|_{L^2}
=0. \label{eq_comp_conv_2.y}
\end{align}
For \eqref{eq_comp_conv_1.y}, one can see 
\begin{align*}
\left\| \tilde{f}^{l_R,R} - f_\infty \right\|_{L^2}^2
&=\sum_b\left\| \tilde{f}^{l_R,b,R} - f_\infty^b \right\|_{L^2}^2\\
&\le \sum_b\left( \left\| \tilde{f}^{l_R,b,R} - f_\infty^{b,R} \right\|_{L^2}^2
+ \left\| f_\infty^b|_{B_\infty(u_\infty^b,R)^c} \right\|_{L^2}^2
\right)\\
&\le 2^{-2R} \cdot \# B_k 
+ \sum_b\left\| f_\infty^b|_{B_\infty(u_\infty^b,R)^c} \right\|_{L^2}^2
\to 0
\end{align*}
as $R\to \infty$. 
Next for \eqref{eq_comp_conv_2.y}, we can see 
\begin{align*}
&\quad\ \limsup_{j\to\infty}\left\| 
\Phi_{i(j)}(\tilde{f}^{l_R,R}) - f_{i(j)} \right\|_{L^2}^2\\
&\le \sum_b\left(
\limsup_{j\to\infty} \left\| 
\Phi_{i(j)}(\tilde{f}^{l_R,b,R}) - f_{i(j)}|_{B_{i(j)}(u_b,R)} \right\|_{L^2}^2
+ \left\| f_{i(j)}|_{B_{i(j)}(u_b,R)^c}
\right) \right\|_{L^2}^2\\
&\le 2^{-2R} \cdot \# B_k 
+ \limsup_{j\to\infty}\left\| f_{i(j)}|_{B_{i(j)}(\tilde{B}_k,R)^c} \right\|_{L^2}^2. 
\end{align*} 
By Proposition \ref{thm_loc_bs. y}, 
for any $\varepsilon>0$ 
there is $R_\varepsilon>0$ such that 
\[
	\| f_i\|_{L^2(B_i(\tilde{B}_k,R_\varepsilon)^c)}^2
	<\varepsilon \| f_i\|_{L^2(S)}^2,
\] 
where $\tilde{B}_k := \{u_b \ | \ b \in B_k \} \subset S$, 
which gives 
\begin{align*}
\limsup_{j\to\infty}\left\| 
\Phi_{i(j)}(\tilde{f}^{l_{R_\varepsilon},R_\varepsilon}) - f_{i(j)} \right\|_{L^2}^2
\le 2^{-2R_\varepsilon} \cdot \# B_k 
+ \varepsilon\sup_{i}\left\| f_i 
\right\|_{L^2}^2. 
\end{align*} 
By taking $\varepsilon \to 0$, we obtain \eqref{eq_comp_conv_2.y}. 

So we see the strong convergence 
$f_{i(j)}\to f_\infty$ as $j\to \infty$. 
Since each of $f_{i(j)}$ are $S^1$-equivariant, 
$f_\infty$ is also $S^1$-equivariant, 
hence $f_\infty\in H_\infty$. 
\end{proof}

Now we can prove the main theorem. 

\begin{dfn}\label{def_alm_lap.y}
\normalfont
Let $(X^{2n}, \omega)$ be a closed symplectic manifold and $(L, \nabla, h)$ be a prequantum line bundle. 
Let $J$ be a compatible almost complex structure on $X$. 
For $k \in \Z_{>0}$, define
\[
\Delta_J^{\sharp k} := \nabla_k^* \nabla_k - kn \colon \Gamma(L^k) \to  \Gamma(L^k), 
\] 
where
\[
\nabla_k \colon \Gamma(L^k) \to \Omega^1(L^k)
\]
is the connection on $L^k$ induced by $\nabla$. 
We have $\Delta_J^{\sharp k} = 2\Delta_{\bar{\del}_J}^k$ when $J$ is an integrable complex structure, 
as we have already mentioned in \eqref{eq_delbarlap.y}. 
\end{dfn}

\begin{thm}\label{thm_main_intro2.h}
Let $(X,\omega)$ be a closed 
symplectic manifold of dimension $2n$, $(L,\nabla,h)$ be a prequantum line bundle and $k \ge 1$ be a positive integer. 
Assume that we are given a non-singular Lagrangian fibration $\mu\colon X \to B$. 
Consider any asymptotically semiflat family of $\omega$-compatible 
almost complex structures $\{J_s\}_{s > 0}$ and assume that 
there is $\kappa\in\R$ with ${\rm Ric}_{g_{J_s}}\ge \kappa g_{J_s}$. 
Then we have a compact convergence of spectral structures 
\begin{align*}
 (L^2(X,L^k),\Delta^{\sharp k}_{J_s}) \xrightarrow{s \to 0}\bigoplus_{b\in B_k} \left( H^k,  \Delta_{\R^n}^k\right).
\end{align*}
in the sense of Kuwae-Shioya \cite{KuwaeShioya2003}. 
\end{thm}

\begin{proof}
By the identification \eqref{eq_delbarlap.y}, it suffices to show 
Theorem \ref{thm_main.y}. 
Take any sequence of positive numbers $s_i >0$ such that $\lim_{i \to \infty}s_i = 0$. 
Let $\Sigma_i$ be the spectral 
structure given by 
$\Delta_{\hat{g}_{J_{s_i}}}^{\rho^k}$. 
It is enough to show that $\Sigma_i \to \Sigma_\infty$ compactly. 
By Proposition \ref{prop_Ric_of_S.h}, we know that uniform lower bound of the Ricci curvatures of $\{(S, \hat{g}_{J_s})\}_{0 < s < \delta }$ are given by the assumption 
${\rm Ric}_{g_{J_s}}\ge \kappa g_{J_s}$. 
So by Proposition \ref{prop_strong_conv.y}, we see that $\Sigma_i \to \Sigma_\infty$ strongly. 
By Fact \ref{fact_KS_conv.h}, 
we need to show that, for any $\{ u_i\}_i$ with 
$\limsup_{i\to\infty}(\| u_i\|_{H_i}^2 
+ \| du_i\|_{H_i}^2) < \infty$, 
there exists a strongly convergent subsequence. 
If $u_i \in (C^\infty(S)\otimes \C)^{\rho_k}$ for all $i$, this is true by Proposition \ref{thm_compact_conv.y}. 
In general for not necessarily smooth $\{u_i\}_{i}$, we can approximate $\{u_i\}_i$ by a sequence $\{u'_i\}_{i}$ with $u'_i \in (C^\infty(S)\otimes \C)^{\rho_k}$, $\lim_i \|u_i - u'_i\| = 0$ and $\limsup_{i\to\infty}(\| u'_i\|_{H_i}^2 + \| du'_i\|_{H_i}^2) < \infty$, so we get the result. 
\end{proof}

\begin{proof}[Proof of Theorem \ref{thm_main_intro.y}] 
By \eqref{eq_delbarlap.y}, 
$2\Delta^k_{\delb_{J_s}}
=\Delta^{\sharp k}_{J_s}$ 
if $J_s$ is integrable. 
By the asymptotic semiflatness of $\{ J_s\}_s$, 
we can give the uniform lower bound of ${\rm Ric}_{g_{J_s}}$ 
by Fact \ref{thm_ricci_bound.y}. 
Then we have Theorem \ref{thm_main_intro.y} 
by 
Theorem \ref{thm_main_intro2.h}. 
\end{proof}

As a consequence of Theorem \ref{thm_main_intro2.h}, 
we have the following result.

\begin{cor}\label{cor_spec_limit_alm.h}
Under the assumptions in Theorem \ref{thm_main.y}, 
let $\lambda_s^j$ be the $j$-th eigenvalue ($j \ge 1$) of $\Delta^{\sharp k}_{J_s}$ acting on $L^2(X; L^k)$, counted with multiplicity. 
For $j \ge 1$, let $N(j) \in \Z_{\ge 0}$ be such that the following inequality is satisfied. 
\begin{align*}
    \# B_k\cdot \frac{(N(j) - 1 + n)!}{n!(N(j)-1)!}
    < j \le \# B_k \cdot \frac{(N(j)  + n)!}{n!(N(j))!}.
\end{align*}
Then we have
\[
\lim_{s \to 0}\lambda_s^j = k\cdot N(j). 
\]
In particular, the number of eigenvalues converging to $0$ is equal to $\# B_k$. 
\end{cor}

We obtain Corollary \ref{cor_spec_limit.y} by applying 
Corollary \ref{cor_spec_limit_alm.h} to the case that 
all of $J_s$ are integrable.

\section{Convergence of quantum Hilbert spaces}\label{sec_spec_gap.h}
We have so far proved the compact spectral convergence result of $\Delta_{\bar{\del}_{J_s}}^k$, 
under asymptotically semiflat deformations of 
integrable complex structures. 
However, actually this does not imply that the quantum Hilbert spaces obtained by the K\aaa hler quantizations, $\{H^0(X_{J_s}; L^k)\}_{s > 0}$, converge to the quantum Hilbert space obtained by the real quantization, $\oplus_{b \in B_k}\C$. 
This is because, there may exist a family of eigenvalues $\{\lambda_{s}\}_{s > 0}$ of $\{\Delta_{\bar{\del}_{J_s}}^k\}_{s > 0}$ with $\lambda_s \neq 0$ for all $s>0$ such that $\lambda_s \to 0$ as $s \to 0$. 
In such cases, the dimensions of quantum Hilbert spaces can jump at $s = 0$ as $s \to 0$. 

In this and the following sections, we show that, if $k$ is large enough, the spaces $\{H^0(X_{J_s}; L^k)\}_{s > 0}$ indeed converge to the space $\oplus_{b \in B_k}\ker \Delta_{\R^n}^k$. 
The lower bound of $k$ is given by the Ricci curvatures of $(X, \hat{g}_s)$. 
Moreover, we also consider the case that $J_s$ 
are not integrable, using the small eigenspaces of $\Delta^{\sharp k}_{J_s}$ instead of $H^0(X_{J_s}; L^k)$.

\subsection{Almost K\"ahler quantization}
In this subsection, we consider geometric quantization for symplectic manifolds $(X, \omega)$, which do not necessarily admit a K\aaa hler structure. 

In general, any symplectic manifold admits compatible almost complex structures, and there are several known ways to generalize K\aaa hler quantization to the quantization of symplectic manifolds equipped with almost complex structures. 
In this paper we consider the {\it almost K\aaa hler quantization} introduced by Borthwick and Uribe in \cite[Section 3]{BU1996} 
and we discuss it based on 
\cite{MaMarinescu2002}. 
This is done by generalizing the $\delb$-Laplacian $\Delta_{\bar{\del}_J}^k$ by $\Delta^{\sharp k}_J$ which is defined in Definition 
\ref{def_alm_lap.y}. 

Guillemin and Uribe showed in \cite{GU1988} that 
$\Delta_J^{\sharp k}$ has the spectral gap 
between large eigenvalues and the other finitely many eigenvalues 
if $k$ is sufficiently large, 
and the dimension of the direct sum of the eigenvectors associated 
with the latter eigenvalues is equal to the Riemann-Roch number of $L^k$. 
Here, we will discuss this idea in the case of $k$ is not large 
but $J$ is close to the integrable one in some sense.

Let $TX\otimes \C=T^{1,0}X\oplus T^{0,1}X$ be 
the decomposition into eigenspaces of $J$ 
such that $J|_{T^{1,0}X}=\sqrt{-1}$ and 
$J|_{T^{0,1}X}=-\sqrt{-1}$. 
Denote by $\nabla^{LC}$ the Levi-Civita connection 
of $g_J$, then we obtain the following connections 
\begin{align*}
\nabla^{1,0}&:=\frac{1}{4}(1-\sqrt{-1}J)\circ\nabla^{LC}\circ
(1-\sqrt{-1}J),\\
\nabla^{0,1}&:=\frac{1}{4}(1+\sqrt{-1}J)\circ\nabla^{LC}\circ
(1+\sqrt{-1}J),
\end{align*}
on $T^{1,0}X,T^{0,1}X$, respectively. 
Define $A_2:=\nabla^{LC}-\nabla^{1,0}-\nabla^{0,1}
\in \Gamma(T^*X\otimes T^*X\otimes TX)$, 
then $A_2=0$ if and only if $J$ is integrable. 
From now on, for a tensor 
$T\in \Gamma(T^*X^{\otimes l}\otimes TX^{\otimes m})$ 
and $x\in X$, 
the norm $|T_x|$ is defined naturally by $(g_J)_x$ and 
we put $\| T\|_J:=\sup_{x\in X} |T_x|$. 
Let $E_1,\ldots,E_n\in T_x^{1,0}X$ be an 
orthonormal basis with respect to $g_J$ and 
$E^1,\ldots, E^n\in \Lambda^{1,0} T_x^*X$ be the dual basis. 
 
Put $\Lambda^{0,\bullet}:=\bigoplus_{q=0}^n\Lambda^{0,q}T^*X$. 
Then the Clifford action on $\Lambda^{0,\bullet}\otimes L^k$ 
is defined by 
\begin{align*}
c(E_i):= \sqrt{2}\overline{E}^i\wedge,\quad
c(\overline{E}_i):=-\sqrt{2}\iota_{\overline{E}_i}.
\end{align*}
If we denote by $\nabla^{Cl}$ the Clifford connection 
on $\Lambda^{0,\bullet}\otimes L^k$, 
then the ${\rm Spin}^c$ Dirac operator $D_k$ 
acting on $\Omega^{0,\bullet}( L^k )$ is defined by 
\begin{align*}
D_k:=\sum_{i=1}^{2n}c(e_i)\nabla^{Cl}_{e_i},
\end{align*}
where $e_1,\ldots,e_{2n}$ form the orthonormal basis 
of $T_xX$ with respect to $g_J$. 
In \cite{MaMarinescu2002}, the Clifford connection is written as 
$\nabla^{Cl} = \nabla^{0,1}\otimes 1_{L^k} + 1_{\Lambda^{0,\bullet}}\otimes\nabla_k + A_2'\otimes 1_{L^k}$, 
where $A_2'$ is given by 
\begin{align*}
A_2'=\frac{1}{2}\sum_{i,j}
\left\{ g_J(A_2(\overline{E}_i),\overline{E}_j)
\overline{E}^i\wedge\overline{E}^j\wedge
+g_J(A_2(E_i),E_j)\iota_{\overline{E}_i}\iota_{\overline{E}_j}
\right\}.
\end{align*}

Denote by $R$ the curvature tensor of $\nabla^{LC}$. 
\begin{dfn}\label{def_delta_almost.h}
\normalfont
Let $\delta\ge 0$. 
$J$ is a {\it $\delta$-almost complex structure} 
if it is $\omega$-compatible almost complex structure 
such that 
\begin{align*}
\sup\left\{ 
\sup_{x,i,j} |R(E_i,E_j)|, \,
\sup_{x,i,j} |R(\overline{E}_i,\overline{E}_j)|,\,
\| A_2\|_J^2,\, \| \nabla^{LC}A_2\|_J\right\}
\le \delta.
\end{align*}
\end{dfn}
\begin{rem}
\normalfont
$J$ is a $0$-almost complex structure 
iff $J$ is integrable. 
\end{rem}

Under the assumption 
that $J$ is a $\delta$-almost complex structure 
for $0\le \delta\le 1$, 
if we have an estimate $\|T\|_J\le C\delta$ 
for some constant $C>0$ depending only on $n$, 
we write 
$T=O(\delta)$. 
For $a>0$ and $b\in\R$, if we have 
$-C\delta a\le b\le C\delta a$, then we also write 
$b=O(\delta)a$. 

\begin{lem}
Let $J$ be a $\delta$-almost complex structure. 
Then we have 
\begin{align*}
D_k^2 = \left(\nabla^{Cl}\right)^*\nabla^{Cl} -kn + 
2k\sum_{q=0}^nq\Pi_q +{\rm Ric^\sharp}\otimes 1_{L^k} 
+O(\delta),
\end{align*}
where $\Pi_q\colon \Lambda^{0,\bullet}\otimes L^k
\to \Lambda^{0,q}\otimes L^k$ is the natural projection 
and ${\rm Ric^\sharp}=\sum_{j,k}{\rm Ric}(E_j,\overline{E}_k)
\overline{E}^k\wedge \iota_{\overline{E}_j}$. 
Moreover, if we put $\nabla_{0,1}^{Cl}:=\sum_i \overline{E}^i
\otimes \nabla_{\overline{E}_i}^{Cl}$, 
then we have 
\begin{align*}
D_k^2 = 2\left(\nabla_{0,1}^{Cl}\right)^*\nabla_{0,1}^{Cl} 
+ 2k\sum_{q=0}^nq\Pi_q +2{\rm Ric^\sharp}\otimes 1_{L^k} 
+O(\delta).
\end{align*}
\label{lem_weitzenbock.h}
\end{lem}
\begin{proof}
First of all we can see 
\begin{align*}
D_k^2
&=\left(\nabla^{Cl}\right)^*\nabla^{Cl}
+\frac{1}{2}\sum_{\alpha,\beta}c(e_\alpha)c(e_\beta)F^{\nabla^{Cl}}(e_\alpha,e_\beta),
\end{align*}
for any orthonormal basis $e_1,\ldots,e_{2n}\in T_xX$. 
It is known that if $J$ is integrable, 
then we have 
\begin{align*}
\frac{1}{2}\sum_{\alpha,\beta}c(e_\alpha)c(e_\beta)F^{\nabla^{Cl}}(e_\alpha,e_\beta) 
= -kn + 
2k\sum_{q=0}^nq\Pi_q +{\rm Ric^\sharp}\otimes 1_{L^k}. 
\end{align*}
If $J$ is not integrable, we can also see that the difference between 
the left-hand side and the right-hand side 
are written by the linear combination of 
the coefficients of $A_2^{\otimes 2}, \nabla^{LC}A_2$ 
and $(2,0)$ or $(0,2)$ component of 
the curvature $R$, 
hence we have the first equality. 
We can also obtain it from \cite[Theorem 2.2]{MaMarinescu2002}. 

The second equality is given by 
comparing $\left(\nabla^{Cl}\right)^*\nabla^{Cl}$ 
and $\left(\nabla^{Cl}_{0,1}\right)^*\nabla^{Cl}_{0,1}$. 
We take 
$E_1,\ldots,E_n$ such that $\nabla^{1,0}E_i|_x=\nabla^{0,1}\overline{E}_i|_x=0$. 
Then we can see  $\sum_i(\nabla^{LC}_{\overline{E}_i}E_i)|_x=\sum_i(\nabla^{LC}_{E_i}\overline{E}_i)|_x=0$ 
by the coclosedness of $\omega$ with respect to $g_J$, 
hence 
we have 
\begin{align*}
2\left(\nabla^{Cl}_{0,1}\right)^*\nabla^{Cl}_{0,1}
&=\left(\nabla^{Cl}\right)^*\nabla^{Cl}
-\sum_iF^{\nabla^{0,1}}(E_i,\overline{E}_i)
-k\sum_iF^{\nabla}(E_i,\overline{E}_i)
+O(\delta)\\
&=\left(\nabla^{Cl}\right)^*\nabla^{Cl}
-{\rm Ric}^\sharp\otimes 1_{L^k}
-kn
+O(\delta),
\end{align*}
which gives the second equality. 
\end{proof}

\begin{lem}\label{lem_weitzenbock2.h}
Let $J$ be a $\delta$-almost complex structure. 
We have 
\begin{align*}
\langle D_k^2\varphi,\varphi\rangle_{L^2} = 
\langle \Delta^{\sharp k}_J \varphi,\varphi\rangle_{L^2}  +O(\delta)
\| \varphi\|_{L^2}^2
\end{align*}
for any $\varphi\in \Omega^{0,0}(L^k)$. 
\end{lem}
\begin{proof}
By the definition of $\nabla^{Cl}$, 
we can see that 
$\|\nabla^{Cl}\varphi\|_{L^2}^2
=\|\nabla_k\varphi\|_{L^2}^2+O(\delta)\|\varphi\|_{L^2}^2$. 
Then we obtain the result  
by Lemma \ref{lem_weitzenbock.h}. 
\end{proof}

\begin{thm}\label{thm_spectral_gap.h}
There is a positive constant $C$ depending only on $n$ 
such that for any $k>0,\kappa\in\R$ with 
$k+\kappa>0$ and $\delta>0$, if 
$J$ is a $\delta$-almost complex structure 
with ${\rm Ric}_{g_J} \ge (\kappa-\delta) g_J$, 
then we have 
\begin{align*}
{\rm Spec}(\Delta^{\sharp k}_J)\subset (-C\delta,C\delta)
\cup (2k+2\kappa - C\delta, +\infty).
\end{align*}
\end{thm}
\begin{proof}
The idea of the proof is essentially based on 
\cite{MaMarinescu2002}. 
Take $\lambda\in\R$ 
and $\varphi\in C^\infty(L^k)$ with 
$\Delta^{\sharp k}_J\varphi=\lambda\varphi$ 
and $\| \varphi\|_{L^2}=1$. 
By Lemma \ref{lem_weitzenbock2.h}, 
we can see 
$\lambda \|\varphi\|_{L^2}^2=\| D_k\varphi\|_{L^2}^2+O(\delta)$, 
hence we obtain $\lambda\ge O(\delta)$. 
Next we have 
\begin{align*}
\langle D_k^2(D_k\varphi),D_k\varphi\rangle_{L^2}
=\| D_k^2\varphi\|_{L^2}^2
\le(\lambda+O(\delta))^2\| \varphi\|_{L^2}^2.
\end{align*}
Moreover, since $D_k\varphi\in\Omega^{0,{\rm odd}}(L^k)$, 
then the second equation of 
Lemma \ref{lem_weitzenbock.h} implies 
\begin{align*}
\langle D_k^2(D_k\varphi),D_k\varphi\rangle_{L^2}
&\ge \| \nabla^{Cl}_{0,1}D_k\varphi\|_{L^2}^2
+(2k+2\kappa+O(\delta))\| D_k\varphi\|_{L^2}^2\\
&\ge (2k+2\kappa+O(\delta))(\lambda+O(\delta))\| \varphi\|_{L^2}^2
\end{align*}
if $\delta$ is sufficiently small such that $k+\kappa+O(\delta)>0$.  
Therefore, there exists a constant $C>0$ depending only on 
$n$ such that 
\begin{align*}
\lambda^2+(3C\delta-2k-2\kappa)\lambda
+(2k+2\kappa)C\delta\ge 0
\end{align*}
if $2k+2\kappa-C\delta>0$, 
which implies $\lambda\le C\delta$ or 
$\lambda\ge 2k+2\kappa-C\delta$ 
by taking $C$ larger. 
If $2k+2\kappa-C\delta\le 0$, then 
we also have the conclusion since 
$(-C\delta,C\delta)
\cup (2k+2\kappa - C\delta, +\infty)
=(-C\delta,+\infty)$. 
\end{proof}

Now, define the Riemann-Roch number 
by 
\begin{align*}
RR(X,L^k):={\rm ind}(D_k)=\dim {\rm Ker}D_k^+
-\dim {\rm Ker}D_k^-,
\end{align*}
where 
$D_k^+=D_k|_{\Omega^{0,{\rm even}}(L^k)}$ 
and $D_k^-=D_k|_{\Omega^{0,{\rm odd}}(L^k)}$. 
The Riemann-Roch number is independent of $J$. 

Let $C$ be a constant appearing in 
Theorem \ref{thm_spectral_gap.h}. 
For a constant $\delta>0$, 
put 
\begin{align*}
\mathcal{H}_{k,\delta}:=
{\rm span}_\C\left\{ \varphi\in C^\infty(L^k)|\,
\exists \lambda\in (-C\delta,C\delta),\,
\Delta^{\sharp k}_J\varphi=\lambda\varphi\right\}.
\end{align*}
\begin{thm}\label{thm_rr_number.h}
Let $k$ be a positive integer and 
$\kappa\in \R$ be a constant such that $k+\kappa>0$. 
There is a constant $\delta_{n,k,\kappa}>0$ 
depending only on $n,k,\kappa$ such that 
for any $\delta\le \delta_{n,k,\kappa}$ 
and any $\delta$-almost complex structure $J$ 
with ${\rm Ric}_{g_J}\ge(\kappa-\delta) g_J$, 
we have $C\delta<2k+2\kappa-C\delta$ and 
\begin{align*}
RR(X,L^k)=\dim\mathcal{H}_{k,\delta}.
\end{align*}
\end{thm}
\begin{proof}
We follow the argument in 
\cite{MaMarinescu2002}. 
Let $J$ be a $\delta$-almost complex structure 
with ${\rm Ric}_{g_J}\ge (\kappa-\delta) g_J$. 
Since $C^{\infty}(L^k)\subset \Omega^{0,\bullet}(L^k)$, 
we regard $\mathcal{H}_{k,\delta}$ as a subspace 
of $\Omega^{0,\bullet}(L^k)$. 
Denote by $P\colon {\rm Ker}(D_k)\to
\mathcal{H}_{k,\delta}$ and $Q\colon\mathcal{H}_{k,\delta}
\to {\rm Ker}(D_k)$ the orthogonal projections 
with respect to the $L^2$ inner product. 
It suffices to show that both of $P,Q$ are injective. 

Note that $D_k^2$ preserves the decomposition 
$\Omega^{0,\bullet}(L^k)=\Omega^{0,{\rm even}}(L^k)
\oplus\Omega^{0,{\rm odd}}(L^k)$. 
By the second equation in Lemma 
\ref{lem_weitzenbock.h}, we can see that 
${\rm Spec}(D_k^2|_{\Omega^{0,{\rm odd}}(L^k)})
\subset [2(k+\kappa)+O(\delta),+\infty)$. 
Next we show that 
${\rm Spec}(D_k^2|_{\Omega^{0,{\rm even}}(L^k)})\subset \{ 0\}\cup (2k+2\kappa+O(\delta),+\infty)$. 
Let $\lambda\in\R$ and $\varphi\in\Omega^{0,\bullet}(L^k)$ 
satisfy $D_k^2\varphi=\lambda\varphi$ and 
suppose $\lambda\neq 0$. 
Then we have $D_k^2(D_k\varphi)=\lambda D_k\varphi$ 
and $D_k\varphi\in \Omega^{0,{\rm odd}}(L^k)$, 
then $\lambda\ge 2(k+\kappa)+O(\delta)$. 

Now, let $\varphi\in\mathcal{H}_{k,\varepsilon}$ and $Q\varphi=0$. 
Then we have 
\begin{align*}
\langle D_k^2\varphi,\varphi\rangle_{L^2}\le 
(C\delta+O(\delta))\| \varphi\|_{L^2}^2
\end{align*}
by Lemma \ref{lem_weitzenbock2.h}. 
Since $Q\varphi=0$ implies that $\langle D_k^2\varphi,\varphi\rangle_{L^2}\ge (2k+2\kappa+O(\delta))\| \varphi\|_{L^2}^2$ 
by the above argument, 
we obtain 
\begin{align*}
2k+2\kappa\le O(\delta)
\end{align*}
for $\varphi\neq 0$. 
Therefore, if we assume $\delta$ 
is sufficiently small, 
then $\varphi=0$. 

Next we take $\varphi\in{\rm Ker}(D_k)$ such that 
$P\varphi=0$. 
Put $\varphi=\sum_q\varphi_q$, where 
$\varphi_q\in\Omega^{0,q}(L^k)$. 
If we put $\psi=\sum_{q\ge 1}\varphi_q$, 
then the second formula of 
Lemma \ref{lem_weitzenbock.h} gives 
\begin{align*}
0=\langle D_k^2\varphi,\varphi\rangle_{L^2}
\ge \| \nabla^{Cl}_{0,1}\varphi\|_{L^2}^2
+(2k+2\kappa+O(\delta))\| \psi\|_{L^2}^2
+O(\delta)\| \varphi_0\|_{L^2}^2,
\end{align*}
hence we have 
\begin{align}
\| \psi\|_{L^2}^2
\le A\delta
\| \varphi_0\|_{L^2}^2\label{eq_thm_rr_number1.h}
\end{align}
for some constant $A$ depending only on 
$n,k,\kappa$, if $\delta$ is sufficiently small. 
In the following argument, replace $A$ by the larger one 
if it is necessary. 
Moreover the first formula of Lemma \ref{lem_weitzenbock.h} gives 
\begin{align*}
0&=\langle D_k^2\varphi,\varphi\rangle_{L^2}\\
&\ge \| \nabla^{Cl}\varphi\|_{L^2}^2
-nk\| \varphi\|_{L^2}^2
+(2k+\kappa+O(\delta))\| \psi\|_{L^2}^2
+O(\delta)\| \varphi_0\|_{L^2}^2.
\end{align*}
Here, we may suppose $2k+\kappa+O(\delta)>0$. 
Then by combining \eqref{eq_thm_rr_number1.h}, 
\begin{align}
\| \nabla^{Cl}\varphi\|_{L^2}^2
\le nk(1+A\delta)\| \varphi_0\|_{L^2}^2.\label{eq_thm_rr_number2.h}
\end{align}
Now, we may write 
$\nabla^{Cl}\varphi=\nabla_k\varphi_0+A_2'\varphi_2+\alpha$, 
where $\alpha$ is the higher degree term. 
Then we have $\|\nabla^{Cl}\varphi\|_{L^2}
\ge \|\nabla_k\varphi_0+A_2'\varphi_2\|_{L^2}
\ge |\|\nabla_k\varphi_0\|_{L^2}
-\|A_2'\varphi_2\|_{L^2}|$, hence 
by \eqref{eq_thm_rr_number1.h}\eqref{eq_thm_rr_number2.h}, 
we have 
\begin{align*}
\| \nabla_k\varphi_0\|_{L^2}
&\le \|\nabla^{Cl}\varphi\|_{L^2}+O(\sqrt{\delta})\|\psi\|_{L^2}\\
&\le \sqrt{nk}(\sqrt{1+A\delta}
+A\delta)\| \varphi_0\|_{L^2},
\end{align*}
which gives 
\begin{align*}
\| \nabla_k\varphi_0\|_{L^2}^2-nk\| \varphi_0\|_{L^2}^2
\le A\delta\| \varphi_0\|_{L^2}^2
\end{align*}
if $\delta$ is sufficiently small. 
This means $\langle \Delta^{\sharp k}_J\varphi_0,\varphi_0\rangle_{L^2}
\le A\delta\| \varphi_0\|_{L^2}^2$, 
therefore, if we take $\delta>0$ such that 
$A\delta \le 2k+2\kappa-C\delta$, 
then $P\varphi=0$ implies $\varphi=0$ 
by Theorem \ref{thm_spectral_gap.h}. 
\end{proof}

\subsection{Convergence of the quantum 
Hilbert spaces}\label{subsec_quantum.h}

Now, fix $k,\kappa$ with $k+\kappa>0$ 
and let $C,\delta_{n,k,\kappa}$ be as in 
Theorems \ref{thm_spectral_gap.h} and 
\ref{thm_rr_number.h}.

\begin{thm}\label{thm_conv_hilb.y}
Let $(X,\omega,L,h,\nabla)$ be a closed 
symplectic manifold 
with a prequantum bundle, 
and let $J_s$ be a family of 
$\delta_{n,k,\kappa}$-almost complex structures with 
${\rm Ric}_{g_{J_s}}\ge (\kappa-\delta_{n,k,\kappa}) g_{J_s}$. 
Define 
\begin{align*}
\mathcal{H}_{k,s}:={\rm span}
\left\{ \varphi\left|\, \exists\lambda\in(-C\delta_{n,k,\kappa},C\delta_{n,k,\kappa}),\, 
\Delta^{\sharp k}_{J_s}\varphi=\lambda\varphi
\right.\right\}
\end{align*} 
and denote by $P_{k,s}\colon L^2(X,g_{J_s}; L^k)
\to \mathcal{H}_{k,s}$ the orthogonal projection. 
We also consider ${\rm Ker}(\Delta_{\R^n}^k)\subset H^k$ 
and denote by $P_k$ the orthogonal projection onto 
this space. 
Then, under the convergence of Hilbert spaces $L^2(X_{J_s}; L^k) \to \oplus_{b \in B_k} H^k$ as $s \to 0$, we have a compact convergence
\begin{align*}
    P_{k, s} \xrightarrow{s \to 0} \oplus_{b \in B_k} P_k, 
\end{align*}
as a family of bounded operators on this family. 
In particular, for such $k$ we have
\begin{align*}
    RR(X,L^k) = \# B_k. 
\end{align*}
\end{thm}
\begin{proof}
Let us denote the spectral projection for $(L^2(X,g_{J_s}; L^k), \Delta^{\sharp k}_{J_s})$ by $E_{s}$ and for $\oplus_{b \in B_k}(H^k, \Delta_{\R^n}^k)$ by $E_\infty$. 
By the compact spectral convergence in Theorem \ref{thm_main_intro2.h} and the definition of compact convergence in Definition \ref{def_spec_conv.y}, we have $E_s((-\varepsilon, \varepsilon]) \to E_\infty((-\varepsilon, \varepsilon])$ compactly as $s \to 0$ for 
any $\varepsilon>0$. 
Put $\varepsilon=C\delta_{n,k,\kappa}$. 
By Theorems \ref{thm_spectral_gap.h} and 
\ref{thm_rr_number.h}, we see that, for $s>0$ small enough, we have $E_s((-\varepsilon, \varepsilon]) = P_{k,s}$. 
Moreover we also have $E_\infty((-\varepsilon, \varepsilon]) = \oplus_{b \in B_k} P_{k}$. 
Thus we get the desired result. 
\end{proof}

\section{Examples}\label{sec_alm.y}
In this section we apply 
Theorem \ref{thm_conv_hilb.y} to 
the following two cases. 

\subsection{K\aaa hler quantization}
In this subsection, we apply 
Theorem \ref{thm_conv_hilb.y} 
in the case of all of $J_s$ are integrable. 
We assume that $\{ J_s\}_s$ is 
asymptotically semiflat family of 
$\omega$-compatible complex structures. 
Then by Fact \ref{thm_ricci_bound.y}, 
there is $\kappa\in\R$ such that 
${\rm Ric}_{g_{J_s}}\ge \kappa g_{J_s}$. 
Take a positive integer $k$ such that $k+\kappa>0$ 
and let $\mathcal{H}_{k,s}$ be as in Theorem 
\ref{thm_conv_hilb.y}. 
We can apply Theorem \ref{thm_spectral_gap.h} 
to this case for any $\delta>0$, hence 
we can see ${\rm Spec}(\Delta^k_{\delb_{J_s}})\subset 
\{ 0\}\sqcup [2k+2\kappa,+\infty)$. 
Since $\delta_{n,k,\kappa}$ is taken 
such that $C\delta_{n,k,\kappa}
<2k+2\kappa-C\delta_{n,k,\kappa}<2k+2\kappa$, 
then 
$\mathcal{H}_{k,s}={\rm Ker}\Delta^k_{\delb_{J_s}}
=H^0(X_{J_s},L^k)$. 
Therefore, 
we obtain the following conclusion. 
\begin{thm}\label{thm_conv_hilb2.h}
Let $(X,\omega,L,h,\nabla)$ be a closed 
symplectic manifold 
with a prequantum bundle, 
and let $J_s$ be an asymptotically semiflat family of 
$\omega$-compatible complex structures. 
Denote by $P_{k,s}\colon L^2(X,g_{J_s}; L^k)
\to H^0(X_{J_s},L^k)$ the orthogonal projection. 
We also consider ${\rm Ker}(\Delta_{\R^n}^k)\subset H^k$ 
and denote by $P_k$ the orthogonal projection onto 
this space. 
Then for sufficiently large $k>0$, 
under the convergence of Hilbert spaces $L^2(X_{J_s}; L^k) \to \oplus_{b \in B_k} H^k$ as $s \to 0$, we have a compact convergence
\begin{align*}
    P_{k, s} \xrightarrow{s \to 0} \oplus_{b \in B_k} P_k, 
\end{align*}
as a family of bounded operators on this family. 
In particular, for such $k$ we have
\begin{align*}
    RR(X,L^k) = \# B_k. 
\end{align*}
\end{thm}

\subsection{Almost K\"ahler quantization}
Here, we show an example of 
one parameter families of $\omega$-compatible almost complex structures $\{J_s\}_{0<s<\delta }$ on $(X, \omega)$ to which we can apply Theorems \ref{thm_main.y} 
and \ref{thm_conv_hilb.y}. 
We assume the following condition on $\{J_s\}_{s}$. 
\begin{itemize}
\setlength{\parskip}{0cm}
\setlength{\itemsep}{0cm}
 \item[$\heartsuit$] 
In the local description as in (\ref{eq_local_hol_frame.y}), the coefficient $A_{ij}$ does not depend on the fiber coordinate $\theta$ and linear in $s$, i.e., we have a local frame of $T^{1,0}_{J_s}X$ of the form
\[
\frac{\partial}{\partial\theta^i}
+sA_{ij}(x) \frac{\partial}{\partial x_j},\quad 
i=1,\ldots,n
\]
for some $A_{ij} \in C^\infty(U) \otimes M_n(\C)$. 
\end{itemize}
Note that this condition is independent of the choice of action-angle coordinate. 
Obviously, $\{ J_s\}_s$ is asymptotically semiflat. 

\begin{prop}\label{prop_alm_ric.y}
Fix $\delta>0$ arbitrarily. 
If a family of almost complex structure $\{J_s\}_{s>0}$ satisfies the condition $\heartsuit$, there exist $s_\delta>0$ such that $J_s$ is a $\delta$-almost 
complex structure and 
$\mathrm{Ric}_{g_{J_s}} \ge -\delta g_{J_{s}}$ for all $0<s<s_0$. 
\end{prop}
\begin{proof}
This follows by straightforward estimates for curvature 
tensors and some other quantities as follows. 
We take a finite covering of $X$ by open sets with action-angle coordinates $(x_1, \cdots, x_n, \theta^1, \cdots, \theta^n)$. 
Put 
\begin{align*}
E_i^s&:=\frac{1}{\sqrt{s}}
\left( \frac{\partial}{\partial\theta^i}
+sA_{ij}(x) \frac{\partial}{\partial x_j}
\right),\\
P_{ij}&={\rm Re}(A_{ij}),\quad
Q_{ij}={\rm Im}(A_{ij}).
\end{align*}
By (\ref{eq_met.y}) and the condition $\heartsuit$, the metric tensor satisfies the condition
\begin{align*}
    g_{J_s}\left( E_i^s,\overline{E}_j^s\right)
= 2Q_{ij},\quad 
g_{J_s}\left( E_i^s,E_j^s\right)
=g_{J_s}\left( \overline{E}_i^s,\overline{E}_j^s\right)=0
\end{align*}
From now on we write 
$E_{\bar{i}}^s:=\overline{E}_i^s$ 
and let 
$\alpha,\beta,\gamma,\tau=1,\ldots,n,\bar{1},\ldots,\bar{n}$. 
Since 
\begin{align*}
g_{J_s}([E_\alpha^s,E_\beta^s],E_\gamma^s)
&= \sqrt{s}g_{J_1}([E_\alpha^1,E_\beta^1],E_\gamma^1),\\
E_\alpha^s\left( 
g_{J_s}(E_\beta^s,E_\gamma^s)\right)
&=\sqrt{s}E_\alpha^1\left( g_{J_1}(E_\beta^1,E_\gamma^1)\right),
\end{align*}
we can see 
\begin{align*}
\nabla^{LC}_{E_\alpha^s}E_\beta^s
&= \sqrt{s}\Gamma_{\alpha\beta}^\gamma(x) E_\gamma^s
\end{align*}
for some function $\Gamma_{\alpha\beta}^\gamma(x)$ 
depending only on $x$, 
where $\nabla^{LC}$ is the Levi-Civita connection 
of $g_{J_s}$. 
Moreover 
we have 
\begin{align*}
g_{J_s}\left( \nabla^{LC}_{E_\alpha^s}\nabla^{LC}_{E_\beta^s}E_\gamma^s, E_\tau^s\right)
&= sg_{J_1}\left( \nabla^{LC}_{E_\alpha^1}\nabla^{LC}_{E_\beta^1}E_\gamma^1, E_\tau^1\right).
\end{align*}
Therefore, 
we can see that all of the geometric quantities 
appearing in Definition \ref{def_delta_almost.h} 
are bounded by $Cs$ from the above, 
where $C$ is a positive constant independent of $s$. 
Since the curvature tensor is bounded, 
the uniform lower bound of the Ricci curvature 
also exists. 
\end{proof}
Thus we may apply 
Theorems \ref{thm_main.y} and 
\ref{thm_conv_hilb.y}
to $\{ J_s\}_s$ 
with $\heartsuit$. 
In particular, 
Theorem \ref{thm_conv_hilb.y} 
can be applied by putting $\kappa=0$, 
hence for any positive $k$.

\section*{Acknowledgment}
The authors are grateful to Shouhei Honda for useful discussions. 
K.Hattori is supported by Grant-in-Aid for 
Scientific Research (C) 
Grant Number 19K03474.
M. Yamashita is supported by Grant-in-Aid for JSPS Fellows Grant Number 19J22404.

\bibliographystyle{plain}

\begin{comment}
\subsection{Ricci form}

\begin{align*}
\end{align*}
\begin{prop}
\end{prop}
\begin{proof}
\end{proof}
\begin{dfn}
\normalfont
\end{dfn}
\begin{rem}
\normalfont
\end{rem}
\begin{lem}
\end{lem}
\begin{thm}[\cite{}]
\end{thm}
\Big(\Big)   \Big\{\Big\}

\[ A=
\left (
\begin{array}{ccc}
 &  &  \\
 &  & \\
 &  & 
\end{array}
\right ).
\]

\begin{itemize}
\setlength{\parskip}{0cm}
\setlength{\itemsep}{0cm}
 \item[$(1)$] 
 \item[$(2)$] 
\end{itemize}

\[ 
\left.
\begin{array}{ccc}
\R^n/{\rm Ker}\,\Phi & \rightarrow & T^n \\
\rotatebox{90}{$\in$} & & \rotatebox{90}{$\in$} \\
\theta\, {\rm mod}\,{\rm Ker}\,\Phi & \mapsto & \theta\, {\rm mod}\,\Z^n
\end{array}
\right.
\]

\end{comment}

\end{document}